\pdfoutput=1
\documentclass[12pt,reqno]{amsart}
\usepackage[lmargin=1.15in,rmargin=1.15in,tmargin=1in,bmargin=1in]{geometry}

\usepackage{yfonts}
\usepackage{amsthm,amsfonts,amssymb}
\usepackage[utf8]{inputenc}
\usepackage{mathrsfs}
\usepackage{enumitem}
\usepackage[colorlinks,unicode,pdfencoding=auto]{hyperref} 
\hypersetup{linkcolor=blue, citecolor=blue}
\usepackage{bookmark}

\usepackage[initials,nobysame]{amsrefs}

\usepackage[scaled=1.15,mathscr]{urwchancal}

\usepackage[foot]{amsaddr}

\makeatletter
\renewcommand{\email}[2][]{%
  \ifx\emails\@empty\relax\else{\g@addto@macro\emails{,\space}}\fi%
  \@ifnotempty{#1}{\g@addto@macro\emails{\textrm{(#1)}\space}}%
  \g@addto@macro\emails{#2}%
}
\makeatother

\newcommand{\keep}[1]{{\binoppenalty=10000\relpenalty=10000 #1}}
\newcommand{\tops}[2]{\texorpdfstring{#1}{#2}}

\newcommand{\<}{\begin{equation}}
\renewcommand{\>}{\end{equation}}

\renewcommand{\epsilon}{\varepsilon}

\def\Xint#1{\mathchoice
   {\XXint\displaystyle\textstyle{#1}}%
   {\XXint\textstyle\scriptstyle{#1}}%
   {\XXint\scriptstyle\scriptscriptstyle{#1}}%
   {\XXint\scriptscriptstyle\scriptscriptstyle{#1}}%
   \!\int}
\def\XXint#1#2#3{{\setbox0=\hbox{$#1{#2#3}{\int}$}
     \vcenter{\hbox{$#2#3$}}\kern-.5\wd0}}
\def\dint{\Xint-}

\newcommand{\ldint}[1]{\dint_{#1-i\infty}^{#1+i\infty}}

\renewcommand{\Re}{\operatorname{Re}}

\let\originalleft\left
\let\originalright\right
\renewcommand{\left}{\mathopen{}\mathclose\bgroup\originalleft}
\renewcommand{\right}{\aftergroup\egroup\originalright}

\usepackage{xfrac} 
\newcommand{\nfrac}[2]{#1/#2}
\newcommand{\pfrac}[2]{\left(\frac{#1}{#2}\right)}
\newcommand{\pf}[2]{\left(\frac{#1}{#2}\right)}
\newcommand{\pfr}[2]{\left(\frac{#1}{#2}\right)}
\newcommand{\ptfr}[2]{(\tfrac{#1}{#2})}

\newcommand{\fr}[2]{\frac{#1}{#2}}
\newcommand{\nfr}[2]{#1/#2}
\newcommand{\tf}[2]{\tfrac{#1}{#2}}
\newcommand{\tfr}[2]{\tfrac{#1}{#2}}

\let\midold\mid
\let\nmidold\nmid
\renewcommand{\mid}{\hspace{-0.2em}\midold\hspace{-0.2em}}
\renewcommand{\nmid}{\hspace{-0.15em}\nmidold\hspace{-0.15em}}

\newcommand{\ssmid}{|}
\newcommand{\ssnmid}{\hspace{0.1em}\nmidold\hspace{0.1em}}

\let\mod\undefined
\newcommand{\mod}[1]{\,(#1)}
\newcommand{\mmod}[1]{\,\,(\mathrm{mod}\,\,#1)}

\newcommand{\lf}{\left}

\newcommand{\rh}{\right}

\newtheorem{theorem}{Theorem}
\newtheorem{proposition}[theorem]{Proposition}
\newtheorem{lemma}[theorem]{Lemma}
\newtheorem{corollary}[theorem]{Corollary}

\theoremstyle{definition}
\newtheorem{remark}[theorem]{Remark}

\numberwithin{theorem}{section}
\numberwithin{equation}{section}
\numberwithin{figure}{section}

\DeclareMathSymbol{\mdot}{\mathord}{symbols}{"01}

\newcommand{\expp}[1]{\exp\left( #1 \right)}
\newcommand{\expb}[1]{\exp\left[ #1 \right]}

\newcommand{\e}[1]{e\left(#1\right)}
 
\newcommand{\floor}[1]{\lfloor#1\rfloor}

\newcommand{\pth}[1]{\left(#1\right)}

\newcommand{\ABS}[1]{\left|#1\right|}
\renewcommand{\O}[1]{O\left(#1\right)}

\newcommand{\rt}[1]{\sqrt{#1}}

\newcommand{\logX}{\log X}
\newcommand{\logx}{\log x}

\newcommand{\signs}{\{0,\pm1\}}

\newcommand{\less}{\ll}
\newcommand{\great}{\gg}

\newcommand{\cc}{\mathbb{C}}

\newcommand{\nn}{\mathbb{N}}

\newcommand{\rr}{\mathbb{R}}
\newcommand{\zz}{\mathbb{Z}}

\newcommand{\fa}{\mathfrak{a}}
\newcommand{\fb}{\mathfrak{b}}
\newcommand{\fc}{\mathfrak{c}}
\newcommand{\fd}{\mathfrak{d}}

\newcommand{\fm}{\mathfrak{m}}
\newcommand{\fp}{\mathfrak{p}}

\newcommand{\fA}{\mathscr{A}} 

\newcommand{\fM}{\mathfrak{M}}
\newcommand{\fP}{\mathfrak{P}}
\newcommand{\fS}{\mathfrak{S}}
\newcommand{\fT}{\mathfrak{T}}

\newcommand{\xa}{\alpha}
\newcommand{\xb}{\beta}
\newcommand{\xc}{\chi}
\renewcommand{\xd}{\delta} 
\newcommand{\xe}{\epsilon}
\newcommand{\xf}{\phi}
\newcommand{\xg}{\gamma}

\newcommand{\xk}{\kappa}
\newcommand{\xl}{\lambda}
\newcommand{\xm}{\mu}
\newcommand{\xn}{\nu}
\newcommand{\xo}{\omega}
\newcommand{\xq}{\theta}
\newcommand{\xr}{\rho}
\newcommand{\xs}{\sigma}
\newcommand{\xt}{\tau}
\newcommand{\xvf}{\varphi}

\newcommand{\xy}{\psi}
\newcommand{\xz}{\zeta}

\newcommand{\xG}{\Gamma}
\newcommand{\xD}{\Delta}
\newcommand{\xF}{\Phi}
\newcommand{\xL}{\Lambda}
\newcommand{\xQ}{\Theta}

\newcommand{\xY}{\Psi}

\newcommand{\cQ}{\mathcal{Q}}
\newcommand{\cR}{\mathcal{R}}

\newcommand{\cX}{\mathcal{X}}

\renewcommand{\l}{\ell}

\usepackage{fancyhdr}
\usepackage{tikz,caption}
\usetikzlibrary{calc,decorations.markings}

\def\centerarc[#1](#2)(#3:#4:#5)
    { \draw[#1] ($(#2)+({#5*cos(#3)},{#5*sin(#3)})$) arc (#3:#4:#5); }

\makeatletter
\def\section{%
    \@startsection{section}{1}%
    \z@{.7\linespacing\@plus\linespacing}{.5\linespacing}%
    {\normalfont\large\bfseries}%
}
\makeatother

\makeatletter
\def\@seccntformat#1{%
  \protect\textup{\protect\@secnumfont
    \csname the#1\endcsname
\space\space
  }%
}
\makeatother

\AtBeginDocument{%
   \def\MR#1{}
}

\begin{document}

\title{Legendre-signed partition numbers} 
\author[T.\ Daniels]{Taylor Daniels}
\address{Dept.\ of Mathematics, Purdue Univ., 150 N University St, W Lafayette, IN 47907}
\email{daniel84@purdue.edu}
\begin{abstract}
    Let $f:\mathbb{N}\to\{0,\pm1\}$, for $n \in \mathbb{N}$ let $\Pi[n]$ be the set of partitions of $n$, and for all partitions $\pi = (a_1,a_2,\ldots,a_k) \in \Pi[n]$ let
    \[
        f(\pi) := f(a_1)f(a_2) \cdots f(a_k).
    \]
    With this we define the $f$-\emph{signed partition numbers}
    \[
        \mathfrak{p}(n,f) = \sum_{\pi\in\Pi[n]} f(\pi).
    \]
    
    In this paper, for odd primes $p$ we derive asymptotic formulae for $\mathfrak{p}(n,\chi_p)$ as $n\to\infty$, where $\chi_p(n)$ is the Legendre symbol $(\frac{n}{p})$ associated to $p$. A similar asymptotic formula for $\mathfrak{p}(n,\chi_2)$ is also established, where $\chi_2(n)$ is the Kronecker symbol $(\frac{n}{2})$. Special attention is paid to the sequence $(\mathfrak{p}(n,\chi_5))_\mathbb{N}$, and a formula for $\mathfrak{p}(n,\chi_5)$ supporting the recent discovery that $\mathfrak{p}(10j+2,\chi_5)=0$ for all $j\geq 0$ is discussed. As a corollary, our main results imply that the periodic vanishing displayed by $(\mathfrak{p}(n,\chi_5))_\mathbb{N}$ does not occur in any sequence $(\mathfrak{p}(n,\chi_p))_\mathbb{N}$ for $p \neq 5$ such that $p\not\equiv 1\,\,(\mathrm{mod}\,8)$.

    In addition, work of Montgomery and Vaughan on exponential sums with multiplicative coefficients is applied to establish a uniform upper bound on certain doubly infinite series involving multiplicative functions $f$ with $|f| \leq 1$.
\end{abstract}

\maketitle

\section{Introduction}
\label{sec:intro}
The \emph{partitions} of a given $n \in \nn$ are the tuples $(a_1,\ldots,a_k)$ of positive integers such that $a_1+\cdots+a_k=n$ and $a_1 \geq \cdots a_k$. For a fixed $f:\nn\to\signs$ and any partition $\pi = (a_1,a_2,\ldots,a_k)$ of any $n \in \nn$, let
    \[
        f(\pi) := f(a_1)f(a_2)\cdots f(a_k).
    \]
With this we define the $f$-\emph{signed partition numbers}
    \<
        \label{eq:fpDef}
        \fp(n,f) = \sum_{\pi\in\Pi[n]} f(\pi),
    \>
where $\Pi[n]$ is the set of (ordinary) partitions of $n$. The integers $\fp(n,f)$ generalize several classical partition number quantities: With the constant function $1$ the quantities $\fp(n,1)$ are the \emph{ordinary partition numbers}, and with $f = \mathbf{1}_A$ for some $A \subset \nn$ the quantities $\fp(n,\mathbf{1}_A)$ are the $A$-\emph{restricted partition numbers}. The sequences $(\fp(n,f))_\nn$ using the M\"obius-$\xm$ and Liouville-$\xl$ functions for $f$ are examined in \cites{daniels2023mobius,daniels2023biasymp}. 

The term \emph{signed} partition number (in the sense of \eqref{eq:fpDef}) was introduced in \cite{daniels2023mobius} to distinguish the quantities $\fp(n,f)$ from \emph{weighted} partition numbers $\fp_f(n)$; in short, signed- and weighted- partition numbers have generating functions of the forms
    \[
        \prod_{n=1}^\infty (1-f(n)z^n)^{-1} = 1 + \sum_{n=1}^\infty \fp(n,f)z^n \quad\text{and}\quad \prod_{n=1}^\infty (1-z^n)^{-f(n)} = 1 + \sum_{n=1}^\infty \fp_{f}(n)z^n,
    \]
respectively, where $f(n)$ is some integer-valued function. In a recent work of Basak, Robles, and Zaharescu \cite{basak2024exponential}, this distinction between weighted- and signed-partitions is reinforced. In particular, in \cite{daniels2023mobius} an application of the Hardy-Littlewood method yields two asymptotic main terms for $\fp(n,\mu)$, while in \cites{basak2024exponential} similar methods demonstrate the nonexistence of main terms for $\fp_\mu(n)$. Additional recent work on weighted partitions (even if not called as such) may be found in, e.g., \cites{stark2022weighted,berndt2024divisor,das2023semiprimes,bridges2023}.

For odd primes $p$ let $\xc=\xc_p$ denote the Legendre symbol $(\fr{n}{p})$ for $p$, i.e., let
    \[
        \xc_p(n) = \begin{cases}
            +1 & \text{$n$ is a quadratic residue (mod $p$),}\\
            -1 & \text{$n$ is a quadratic nonresidue (mod $p$),}\\
            0 & p\mid n.  
        \end{cases}
    \]
In addition let $\xc_2(n)$ be the Kronecker symbol $(\tf{n}{2})$, i.e., let
     \<
        \label{eq:int:Kronecker}
        \xc_2(n) = \begin{cases}
            +1 & n\equiv \pm1 \mmod{8},\\
            -1 & n\equiv \pm3 \mmod{8}, \\
            0 & \text{otherwise}.
        \end{cases}
    \>
We note that in what follows, the letter $p$ is reserved for odd primes, and $\xc=\xc_p$ and $\xc_2$ always indicate the Legendre and Kronecker symbols, respectively.

To serve as a benchmark for our formulae, we recall from \cite{hardy1918asymptotic} Hardy and Ramanujan's formula 
    \[
        \fp(n,1) = (4\rt{3})^{-1}n^{-1}\exp(\xk\rt{n})\lf[1+O(n^{-1/5})\rh],
    \]
where, for the remainder of this section, we let
    \[
        \xk := \pi\rt{2/3}.
    \]
This immediately implies the corollary relation
    \[
        \log \fp(n,1) \sim \xk\rt{n},
    \]
where $a_n \sim b_n$ indicates that $\lim_{n\to\infty} a_n/b_n = 1$. To avoid technicalities due to small values of $n$ for which $\fp(n,\xc_p)=0$, in our statements below we let
    \[
        \log^+ x := \max\{0, \log x\}. 
    \]

\begin{theorem}
    \label{thm:logPnc}
    Suppose that $p\neq 5$ and $p\not\equiv 1\mmod{8}$. As $n\to\infty$ one has
        \<
            \label{eq:logPnc}
            \log^+{\fp(n,\xc_p)} \sim \tf12\xk\rt{(1-\tf1p)n}.
        \> 
\end{theorem}

When $p \equiv 1 \mmod{8}$ and $n$ is even, relation \eqref{eq:logPnc} still holds, but when $n$ is odd there may be large cancellations in the asymptotic main terms of $\fp(n,\xc_p)$. As such, we presently only demonstrate a crude bound for $\fp(2m+1,\xc_p)$ as seen in the following theorem.

\begin{theorem}
    \label{thm:logP18}
    If $p\equiv 1\mmod{8}$, then for even $n$, say $n=2m$, one has
        \[
            \log^{+}{\fp(2m,\xc_p)} \sim \tf12\xk\rt{(1-\tf1p)(2m)},
        \]
    and for odd $n$, say $n=2m+1$, one has
        \[
            \fp(2m+1,\xc_p) = O\Big( m^{-\fr{19}{20}}\exp\lf(\tf12\xk\rt{(1-\tf1p)(2m+1)}\,\rh)\!\Big).
        \]
\end{theorem}

Theorems \ref{thm:logPnc} and \ref{thm:logP18} are corollaries of our main results which we now discuss. Ultimately, the two most salient properties of the prime $p$ that determine the asymptotic behavior of $\fp(n,\xc_p)$ are: i) the residue of $p$ (mod $4$); and ii) whether or not $2$ is a quadratic residue modulo $p$, i.e., the value of $\xc_p(2)$. Thus, we separate odd primes by their residues modulo $8$.

In stating our formulae it is convenient to let
    \[
        L_1(\xc_p) = L(1,\xc_p) \qquad\text{and}\qquad L_1'(\xc_p) = L'(1,\xc_p),
    \]
where $L(s,\xc_p)$ is Dirichlet $L$-function for $\xc_p$. Many of our formulae involve complicated constants, and hence we wish to convey to the reader that the ``structures'' of the formulae are more important than the exact values of the constants therein.

\begin{theorem}
\label{thm:P14}
    Let $p\neq 5$ and suppose that $p\equiv 1\mmod{4}$. As $n\to\infty$ one has
        \[
            \fp(n,\xc_p) = \fa_p n^{-\fr34}\exp\lf(\tfr12\xk\rt{(1-\tf1p)n}\,\rh) \lf[ 1+(-1)^n\fb_p +O(n^{-\fr15}) \rh],
        \]
    where
        \[
            \xk = \pi\rt{\tf23}, \qquad \fa_p = \ptfr{p-1}{384\,p^2}^{\fr14} \exp(\tf14\rt{p}L_1(\xc_p))
        \]
    and
        \[ 
        \fb_p = 
        \begin{cases} 
            1 & p\equiv 1\mmod{8}, \\
            \exp(-\rt{p}L_1(\xc_p)) & \text{$p\equiv 5\mmod{8}$ and $p\neq 5$.}
        \end{cases}
        \]
\end{theorem}

Since $L_1(\xc_p) > 0$ for all primes $p$ (and $L_1(\xc_2) > 0$ as well), see, e.g.\ \cite{montgomery2007multiplicative}*{p.\ 124}, we have the following corollary which provides a ``middle ground'' of complexity between Theorems \ref{thm:logPnc} and \ref{thm:P14}. We recall that the relation $a_n \asymp b_n$ holds if
    \[
        a_n = O(b_n) \quad\text{and}\quad b_n = O(a_n) \quad\text{as $n\to\infty$.}
    \]

\begin{corollary}
    \label{cor:P58asymp}
    Suppose that $p\equiv 5\mmod{8}$ and $p\neq 5$. As $n\to\infty$ one has 
        \[
            \fp(n,\xc_p) \asymp n^{-\fr34}\exp\lf(\tf12\xk\rt{(1-\tf1p)n}\,\rh).
        \]
\end{corollary}

Although the asymptotic formulae for $\fp(n,\xc_p)$ with $p\equiv 3\mmod{4}$ have the same overall structure as those for $p\equiv 1\mmod{4}$, in these cases the constants $\fa_p$ and $\fb_p$ have significantly more complicated formulae.

\begin{theorem}
    \label{thm:P34}
    Suppose that $p\equiv 3\mmod{4}$, let $\xk=\pi\rt{2/3}$, let $\xc=\xc_p$, and let
        \<
            \label{eq:int:P34fa}
            \fa_p = \pfr{p-1}{384\,p^2}^{\fr14} \exp\!\bigg( \fr{\rt{p}L_1(\xc)}{2\pi}\bigg( \xg + \fr12\log\!\bigg(\fr{384}{p-1}\bigg) - \fr{L_1'(\xc)}{L_1(\xc)} \bigg)\!\bigg),
        \>
    where $\xg$ is the Euler-Mascheroni constant.
    If $p\equiv 3\mmod{8}$, then as $n\to\infty$ one has
    \[
		\fp(n,\xc_p) = \fa_p n^{\rt{p}L_1(\xc)/4\pi-3/4} \exp\lf(\tf12\xk\rt{(1-\tf1p)n}\,\rh) \lf[ 1+(-1)^n \fb_p n^{-\rt{p}L_1(\xc)/\pi} + O(n^{-1/5}) \rh],
	\]
    where
    \[
        \fb_p = \exp\lf[\fr{\rt{p}L_1(\xc)}{\pi}\lf(2\fr{L_1'(\xc)}{L_1(\xc)} + \log\pfr{p(p-1)}{192} - 2\xg \rh)\rh].
    \]
    If $p\equiv 7\mmod{8}$, then as $n\to\infty$ one has
    \[
        \fp(n,\xc_p) = \fa_p n^{\rt{p}L_1(\xc)/4\pi-3/4} \exp\lf(\tf12\xk\rt{(1-\tf1p)n}\,\rh) \lf[ 1 + (-1)^n \fb_p + O(n^{-1/5}) \rh],
    \]
    where
    \[
        \fb_p = 2^{-\rt{p}L_1(\xc)/\pi}.
    \]
\end{theorem}

Again because $L_1(\xc_p)>0$ for all $p$, we have the following analogue of Corollary \ref{cor:P58asymp}.

\begin{corollary}
    \label{cor:P34Asymp}
    If $p\equiv 7\mmod{8}$, then as $n\to\infty$ one has
        \[
            \fp(n,\xc_p) \asymp n^{\rt{p}L_1(\xc)/4\pi-3/4} \exp\lf(\tf12\xk\rt{(1-\tf1p)n}\,\rh).
        \]
    If $p\equiv 3\mmod{8}$, then as $n\to\infty$ one has the stronger relation
        \[
            \fp(n,\xc_p) \sim \fa_p n^{\rt{p}L_1(\xc)/4\pi-3/4} \exp\lf(\tf12\xk\rt{(1-\tf1p)n}\,\rh),
        \]
    where $\fa_p$ is again given by \eqref{eq:int:P34fa}.
\end{corollary}

Given the appearance of the term $\rt{p}L_1(\xc)/4\pi-3/4$ in the exponents of $n$ in Theorem \ref{thm:P34}, we pause to consider the potential size of this term. Siegel's theorem (see, e.g., \cite{montgomery2007multiplicative}*{Thm.\ 11.14}) states that: For each $\xe>0$ there exists a constant $C(\xe)>0$ such that: if $\xy$ is a quadratic character modulo $q$, then 
    \[
        L_1(\xy) > C(\xe) q^{-\xe}.
    \]
Although it is well-known that the constant $C(\xe)$ above is ineffective, Siegel's theorem implies that for any $\xe>0$ we have
    \[
        \rt{p}L_1(\xc_p)/4\pi - 3/4 \asymp p^{\nfr12-\xe} \qquad (p\to\infty).
    \]
As such, the terms $n^{\rt{p}L_1(\xc_p)/4\pi - 3/4}$ in Theorem \ref{thm:P34} and Corollary \ref{cor:P34Asymp} may have arbitrarily large exponents as the prime $p$ increases.

\subsection{The special cases of \tops{$\fp(n,\xc_5)$}{p(n,chi5)} and \tops{$\fp(n,\xc_2)$}{p(n,chi2)}}

Examining Theorems \ref{thm:P14} and \ref{thm:P34}, one sees that the formulae for $\fp(n,\xc_p)$ therein all contain a factor of the form
    \<
    \label{eq:int:GenSign}
        1 + (-1)^n\fb_p n^{-\xd} + O(n^{-1/5}) \qquad (\xd \geq 0).
    \>    
In essence, the sequence $(\fp(n,\xc_5))_\nn$ behaves so differently than the other $(\fp(n,\xc_p))_\nn$ because the factor in the formula for $\fp(n,\xc_5)$ corresponding to \eqref{eq:int:GenSign} includes an extra cosine term. We remark that since $p$ is now explicitly set to be $5$, the constants $\fa_p$ and $\fb_p$ (now $\fa_5$ and $\fb_5$) of Theorem \ref{thm:P14} may be explicitly computed.

\begin{theorem}
    \label{thm:P5}
    As $n\to\infty$ one has
    \<
        \label{eq:int:P5Form}
		\fp(n,\xc_5) = \fa_5 n^{-\fr34}\exp\!\Big(\tf12\xk\rt{\tf45n}\,\Big) \lf[ 1 + (-1)^n\fb_5 + \fd_5 \cos\lf(\tf{2\pi}{5}n-\tf{\pi}{10}\rh) + O(n^{-\fr15}) \rh],
    \>
    where
    \[
        \xk = \pi\rt{\tf23}, \quad \fa_5 = \bigg(\fr{3+\rt{5}}{960}\bigg)^{\fr14}, \quad \fb_5 = \fr{3-\rt{5}}{2}, \quad\text{and}\quad \fd_5 = \rt{2(5-\rt{5})}.
    \]
\end{theorem}

Ignoring the error term $O(n^{-1/5})$ in formula \eqref{eq:int:P5Form}, we see that the behavior of $\fp(n,\xc_5)$ is essentially dictated by two core parts, namely the exponential term
    \[
        \fa_5 n^{-\fr34}\exp\!\Big(\tf12\xk\rt{\tf45n}\,\Big)
    \]
and the $10$-periodic ``signed'' term
    \[
        \fS(n) := 1 + (-1)^n\bigg(\fr{3-\rt{5}}{2}\bigg) + \rt{2(5-\rt{5})} \, \cos\lf(\fr{2\pi n}{5}-\fr{\pi}{10}\rh).
    \]
Computing the values of $\fS(n)$ for $1 \leq n \leq 10$, it is surprising to find that 
    \[
        \fS(2) = 0 \qquad\text{and}\qquad \fS(n)\neq 0 \quad\text{for $1 \leq n \leq 10$ with $n\neq 2$}. 
    \]
Indeed, since
    \[
        \cos\lf(\fr{7\pi}{10}\rh) = -\fr12\rt{\fr{5-\rt{5}}2},
    \]
when $n=2$ we have
    \[
        \fS(2) = \fr{5-\rt{5}}{2} - \fr{\rt{(5-\rt{5})^2}}{2} = 0.
    \]
Although formula \eqref{eq:int:P5Form} is not an exact formula for $\fp(n,\xc_5)$, the fact that $\fS(10j+2)=0$ for $j\geq0$ provides a heuristic explanation for the following surprising result, established in \cite{daniels2024vanishing} using $q$-series methods.

\begin{theorem}[\cite{daniels2024vanishing}*{Thm.\ 1.1}]
    \label{thm:vanishing}
    One has $\fp(10j+2,\xc_5) = 0$ for all $j \geq 0$.
\end{theorem}

We now consider $\fp(n,\xc_2)$, recalling that
    \[
        \xc_2(n) = \begin{cases}
            1 & n\equiv \pm1\mmod{8},\\
            -1 & n\equiv \pm3 \mmod{8}, \\
            0 & \text{otherwise}.
        \end{cases}
    \]
Like our formula for $\fp(n,\xc_5)$, our formula for $\fp(n,\xc_2)$ is unique among those for other $\fp(n,\xc_p)$. Specifically, the dominant asymptotic term for $\fp(n,\xc_2)$ is a single 8-periodic cosine term, rather than the common terms such as $1+(-1)^n\fb_p$ (or similar quantities).

\begin{theorem} 
    \label{thm:Kron}
    As $n\to\infty$, one has
    \[
        \fp(n,\xc_2) = \fa_2 n^{-\fr34} \exp\!\Big(\tf12\xk\rt{\tf{11}{16}n}\,\Big) \fS_2(n),
    \]
    where
    \<
        \label{eq:fS2}
        \fS_2(n) = \cos\lf(\fr{2\pi n}{8}-\fr{3\pi}{16}\rh)
        + \fb_2\lf[1 + \fr{(-1)^n}{1+\rt2}\rh]\exp\!\Big(\!\!-\!\tfr12\xk\rt{\tf12n}\Big[\rt{\tfr{11}{8}}-1\Big]\!\Big),
    \>
    and
    \[
        \xk=\pi\rt{\tf23}, \quad \fa_2 = \pfr{11}{384}^{\fr14}, \quad\text{and}\quad \fb_2 = \rt{\fr{1+\rt2}{2\rt{11}}}.
    \]
\end{theorem}

Motivated by Theorem \ref{thm:vanishing}, we say that a sequence $(a_n)_\nn$ \emph{vanishes on some arithmetic progression} if $a_{mj+r}=0$ for some $m \in \nn$, some $0 \leq r < m$, and all $j \geq 0$; in such a case we may say that $a_n$ ``vanishes for all $n\equiv r\mmod{m}$''. Thus $\fp(n,\xc_5)$ vanishes for all $n\equiv 2\mmod{10}$.

As $n$ runs from $1$ to $8$, the quantity $\cos(\fr{2\pi }{8}n-\fr{3\pi}{16})$ never vanishes; consequently, the signed term $\fS_2(n)$ in \eqref{eq:fS2} never vanishes either, and it follows that $(\fp(n,\xc_2))_\nn$ does not vanish on any arithmetic progression. Through this lens of vanishing on arithmetic progressions, the results of Corollaries \ref{cor:P58asymp} and \ref{cor:P34Asymp} and Theorem \ref{thm:Kron} combine to yield the following interesting result.

\begin{theorem}
    If $p$ is an odd prime such that $p\neq 5$ and $p \not\equiv 1 \mmod{8}$, then the sequence $(\fp(n,\xc_p))_\nn$ does not vanish on any arithmetic progression. Similarly, the sequence $(\fp(n,\xc_2))_\nn$ does not vanish on any arithmetic progression.
\end{theorem}

Although outside the scope of our discussions, we mention another example of a signed partition sequence that vanishes on arithmetic progressions.

\begin{theorem}[\cite{daniels2024vanishing}*{Thm.\ 1.1}]
    For all partitions $\pi = (a_1,a_2,\ldots,a_k)$ of $n \in \nn$, let
        \[
            \xc_5^\dag(\pi) := (-1)^k \xc_5(a_1)\cdots \xc_5(a_k).
        \]
    Then $\fp(n,\xc_5^\dag)=0$ vanishes for all $n \equiv  6 \mmod{10}$.
\end{theorem}

\subsection{A result on exponential sums}
For $f:\nn\to\cc$ with $|f|\leq1$, the generating function $\xF(z) = \xF(z,f)$ for $(\fp(n,f))_\nn$ is defined via
    \[
    	\xF(z) = \prod_{n=1}^\infty \pfr{1}{1 - f(n)z^{n}} = 1 + \sum_{n=1}^\infty \fp(n,f) z^n \qquad (|z|<1).
    \]
With this, by Cauchy's theorem we have
	\<
		\label{eq:pnfInt}
		\fp(n,f) = \fr{1}{2 \pi i}\int_{|z|=\xr} \xF(z) z^{-n-1} \, dz 
			= \xr^{-n} \int_{0}^{1} \xF(\xr e(\xa)) e(-n\xa) \, d\xa
	\>
for all $0<\xr<1$, where $e(\xa):=\exp(2\pi i\xa)$. In practice, one often analyzes the integrand in \eqref{eq:pnfInt} by examining a logarithm of $\xF(z)$ rather than $\xF(z)$ itself. Explicitly, the majority of our analyses consider the series $\xY(z) = \xY(z,f)$ defined via
    \<
        \label{eq:PsiDefin}
        \xY(z) = \sum_{k=1}^\infty \sum_{n=1}^\infty \fr{f^k(n)}{k} z^{nk} \qquad (|z| < 1),
    \>
so that 
    \[
        \xF(z) = \exp(\xY(z)).
    \]

As outlined above, our primary focus is on sequences $(\fp(n,\xc_p))_\nn$ and on functions $\xY(\xr e(\xa),\xc_p)$ for different primes $p$. However, because it may be of independent interest, in section \ref{sec:Minor} we establish the following general result using work of Montgomery and Vaughan \cite{montgomery1977exponential}.

\begin{theorem}
	\label{thm:Minor}
	Let $X$ be a large positive parameter, and let $\xa$ have the property that: if $|\xa-a/q| \leq 1/(qX^{2/3})$ for coprime $q\in\nn$ and $a\in\zz$, then $q > X^{1/3}$. Then, writing $\xr = \exp(-1/X)$, as $X\to\infty$ one has
		\[
			\xY(\xr e(\xa),f) = O(X/\log{X}),
		\]
    uniformly among multiplicative $f$ with $|f|\leq1$.
\end{theorem}

\subsection*{Acknowledgements}
The author would like to thank Trevor Wooley for suggesting this problem, for numerous helpful conversations, and for his financial support during part of this research. Thanks are additionally extended to David McReynolds and Alexandru Zaharescu for invaluable conversations and advice in preparing this paper.

\section{Notation and preliminaries}

The letter $p$ is reserved for odd primes, the notation $\xc=\xc_p$ is reserved for the Legendre symbol $(\fr{n}{p})$, and $\xc_2$ is reserved for the Kronecker symbol $(\fr{n}{2})$. The notation $\xc_0$ is reserved for principal Dirichlet characters, where the associated modulus is specified in context. To avoid confusion, we emphasize the following convention: Except in statements such as ``the Legendre symbol $(\fr{n}{p})$,'' fractional expressions $(\fr{a}{b})$ are never used to indicate a Legendre or Kronecker symbol; rather, these functions are always indicated by $\xc_p$ (or simply $\xc$) and $\xc_2$, respectively.

Expressions $O(g(x))$ with $g(x) > 0$ denote quantities bounded by some constant multiple of $g(x)$, and $f(x) \less g(x)$ holds if $f(x) = O(g(x))$ as $x\to\infty$. The relation $f(x)\asymp g(x)$ holds when $f(x) \ll g(x)$ and $g(x) \ll f(x)$. Expressions $X_0(A)$ indicate some constant $X_0 > 0$ depending on $A$. Statements involving $\xe$ are understood to hold for all sufficiently small $\xe > 0$ unless noted otherwise, and constants' dependencies on $\xe$ are generally suppressed.  

For real $\xa$ let $e(\xa)=\exp(2\pi i\xa)$, let $\|\xa\| = \min_{n\in\zz}|\xa-n|$, and let $\{\xa\} = \xa - \floor{\xa}$. For $f:\nn\to\cc$ let
    \[
        S_f(t,\xa) = \sum_{n \leq t} f(n)e(n\xa).
    \]
Expressions $n \equiv  a$ (mod $q$) are often abbreviated as $n \equiv  a \mod{q}$ in subscripts, and sums the form $\sum_{n\equiv a\mod{q}}$ or $\sum_{2\ssmid k}$ indicate sums over positive $n$ and $k$ satisfying the conditions $n\equiv a\mmod{q}$ and $2\mid k$, respectively.

Let $1 \leq Q \leq X$. For $q \in \nn$ and $0 \leq a \leq q$ with $(a,q)=1$ we define
	\[
		\fM(a/q) = \lf\{ \xa \in [0,1) : |\xa - a/q| \leq Q/(qX) \rh\}.
	\]
The \emph{major arcs} are those $\fM(\nfr{a}{q})$ with $q \leq Q$, and $\fM(X,Q)$ denotes the union of all major arcs. The connected components of $[0,1) \setminus \fM(X,Q)$ are the \emph{minor arcs}, and their union is denoted $\fm(X,Q)$. In many cases, only $\xa$ lying in small subsets of the arcs $\fM(\tf01)$, $\fM(\tf12)$, and $\fM(\tf11)$ significantly contribute to the integral \eqref{eq:pnfInt}, with all other $\xa$ contributing to an error term. Anticipating this, for $X,X_*>1$ we define
    \begin{align*}
        \fP &= \{ \xa \in [0,1) : \|\xa\| \leq 3/(8\pi X) \}, \\
        \fP_* &= \{\xa \in [0,1) : |\xa - \tf12| \leq 3/(8\pi X_*) \},
    \end{align*}
and, following Gafni \cite{gafni2021partitions}, we term the sets $\fP$ and $\fP_*$ the \emph{principal arcs}.

For $\rho, \rho_* \in (0,1)$ it is convenient to write
    \[
        X = \frac{1}{\log(1/\rho)} \qquad\text{and}\qquad X_* = \frac{1}{\log(1/\rho_*)},
    \]
so that
    \[
        \rho = e^{-1/X} \qquad\text{and}\qquad \xr_* = e^{-1/X_*},
    \]
and this convention is maintained throughout this paper.

By comparison with the series $\xY(z,1)$, for $|f|\leq1$ the series $\xY(z,f)$ in \eqref{eq:PsiDefin} converges absolutely for $|z|<1$. Thus, we exchange the sums in equation \eqref{eq:PsiDefin} and define
    \[
        F_k(z) = \sum_{n=1}^\infty f^k(n) z^n
    \]
for $k > 0$, so that
    \<
        \label{eq:PsiF}
        \xY(z) = \sum_{k=1}^\infty \sum_{n=1}^\infty \fr{f^k(n)}{k} z^{nk} = \sum_{k=1}^\infty \fr{1}{k} F_k(z^k).    
    \>
Often we analyze the sum \eqref{eq:PsiF} by considering the sum of terms with $k \leq K$ for some parameter $K$ and treating the remaining sum over $k > K$ as an error term. For this purpose we record the following lemma bounding these tail sums.

\begin{lemma}
    \label{F:lem:TailBB}
    Let $f:\nn\to\cc$ with $|f|\leq1$. For real $\xa$ and large $X$ and $K$, one has
        \[
            \sum_{k > K} \fr{1}{k} F_k(\xr^k e(k\xa)) \less \fr{X}{K}.
        \]
\end{lemma}

\begin{proof}
    For fixed $k$ we see that
        \[
            F_k(\xr^k e(k\xa)) = \sum_{n=1}^\infty f^k(n) \xr^{nk} e(nk \xa) \less \sum_{n=1}^\infty \xr^{nk} = \fr{\xr^k}{1 - \xr^k}.    
        \]
    Using the relation $\xr = e^{-1/X}$ it follows that
        \[
            \sum_{k > K} \fr{1}{k} F_k(\xr^k e(k\xa)) \less \sum_{k > K} \fr{1}{k} \, \fr{ e^{-k/X} }{1 - e^{-k/X} } = \sum_{k > K} \fr{1/k}{ e^{k/X} - 1 } \less \sum_{k > K} \fr{X}{k^2},    
        \]
    and the result follows.
\end{proof}

\subsection*{Structure}
The majority of this paper focuses on $\fp(n,\xc_p)$ for odd primes $p$, after which the special case of $\fp(n,\xc_2)$ is treated in section \ref{sec:Kron}.
In section \ref{sec:Minor} we analyze $\xY(\xr e(\xa),f)$ for $\xa$ in the minor arcs $\fm(X,Q)$, where $f$ is multiplicative and $|f|\leq1$, ultimately establishing Theorem \ref{thm:Minor}. After this, in sections \ref{sec:major} and \ref{sec:princ} we restrict our attention back to $\xY(z,\xc_p)$ for odd $p$, and therein derive asymptotic formulae for $\Psi(\xr e(\xa),\xc_p)$ as $\xr\to1^{-}$ for $\xa$ in the major arcs and principal arcs. 

In section \ref{sec:relations} we consider how the parameters $X$ and $X_*$ (and thus $\xr$ and $\xr_*$) should depend on $n$. Because the ``optimal'' radii of integration $\xr$ and $\xr_*$ may be unequal, in section \ref{sec:transference} we recall the ``arc transference'' of \cite{daniels2023mobius}*{sec.\ 9} and modify the contour of integration in \eqref{eq:pnfInt}. Although the contents of sections \ref{sec:Minor}--\ref{sec:transference} are applicable to the analysis of $\fp(n,\xc_5)$, in sections \ref{sec:Nonprincipal}--\ref{sec:AsymptoticsMod8} we must exclude the case of $\fp(n,\xc_5)$ until its specific treatment in section \ref{sec:5}. 

Thus, focusing on $\fp(n,\xc_p)$ for odd $p\neq 5$, in section \ref{sec:Nonprincipal} we use the results of the previous sections to establish that the integral for $\fp(n,\xc_p)$ is indeed dominated by those $\xa$ near to $0$, $\tf12$, and $1$, i.e., those $\xa$ in the principal arcs $\fP$ and $\fP_*$. 
Following this, in section \ref{sec:MainInts} we derive asymptotic formulae for these integrals over $\fP$ and $\fP_*$, which yields a ``crude'' asymptotic formula for $\fp(n,\xc_p)$ (again only for odd $p\neq 5$). In sections \ref{sec:Asymptotics} and \ref{sec:AsymptoticsMod8} this formula is refined by considering the possible residues of $p$ modulo $8$, and the formulae of Theorems \ref{thm:P14} and \ref{thm:P34} are established. 

In sections \ref{sec:5} and \ref{sec:Kron} we treat the special cases of $\fp(n,\xc_5)$ and $\fp(n,\xc_2)$, respectively, therein establishing Theorems \ref{thm:P5} and \ref{thm:Kron}, respectively. The first part of the appendix contains a number of elementary but tedious computations used to establish Proposition \ref{prop:Major}. Finally, in the second part of the appendix we include some useful formulae for explicitly computing $L_1(\xc_p)$ and related quantities.

\subsection*{A comment on the scope of this paper}
As discussed in the introduction, our primary focus in this paper is on asymptotic formulae for the Legendre-signed partition numbers $\fp(n,\xc_p)$ for odd primes, as well as a similar formula for $\fp(n,\xc_2)$. However, the majority of the techniques and derivations in this paper may be easily generalized and applied to, say, partition numbers $\fp(n,\xy)$ signed by general Dirichlet characters $\xy$. 

We restrict our focus to $\fp(n,\xc_p)$ and $\fp(n,\xc_2)$ here for the relative ease of the associated analyses and derivations. As the reader is likely to observe in section \ref{sec:intro}, even in these simple cases the formulae for $\fp(n,f)$ can involve unwieldy constants and expressions. In addition, in sections \ref{sec:5} and \ref{sec:Kron}, where $\fp(n,\xc_5)$ and $\fp(n,\xc_2)$ are treated, respectively, one sees that our applications of the Hardy-Littlewood method can noticeably vary even among cases where $f$ is a primitive quadratic character. As such, it is only natural that our attention be restricted here.

\section{The minor arcs}
\label{sec:Minor}

The goal of this section is the proof of Theorem \ref{thm:Minor}, which, by the definitions of $\fM(X,Q)$ and $\fm(X,Q)$, concerns $\xY(\xr e(\xa))=\xY(\xr e(\xa),f)$ for $\xa\in\fm(X,X^{1/3})$. Although the statement of Theorem \ref{thm:Minor} indicates that ultimately $Q:=X^{1/3}$, we begin with an undetermined $Q$ to better illustrate our derivations.

Recalling that $S_f(x,\alpha) = \sum_{n \leq x} f(n)e(n\alpha)$ and $\xY(\xr e(\xa)) = \sum_{k=1}^\infty k^{-1}F_k(\xr e(\xa))$, our plan is to use summation by parts to write
    \[
        F_k(\xr e(\xa)) = \sum_{n=1}^\infty f^k(n)\xr^{nk} e(nk\xa) = \fr{k}{X} \int_0^\infty e^{-xk/X}S_f(x,k\xa) \,dx,
    \]
and then estimate these $F_k$ by employing suitable bounds on $S_f(x,k\xa)$. The key to these bounds is the following result due to Montgomery and Vaughan.

\begin{lemma}[{\cite{montgomery1977exponential}*{Cor.\ 1}}]
	\label{lem:MVCor0}
	Suppose $\lf|\xa - a/q\rh| \leq q^{-2}$, $(a,q)=1$, and $2 \leq R \leq q \leq x/R$. Then
		\<
			\label{eq:MVCorBB0}
			S_f(x, \xa) \less \fr{x}{\log x} + \fr{x\log^{3/2}R}{R^{1/2}}
		\>
	uniformly for all multiplicative $f$ with $|f| \leq1$. In addition, the implicit constant in inequality \eqref{eq:MVCorBB0} does not depend on $q$ or $a$.
\end{lemma}

Prior to applying Lemma \ref{lem:MVCor0}, we must address a subtle issue concerning the quantities $S_f(x,k\xa)$. Namely, as $k$ runs over $\nn$ the quantities $k\xa$ may fall outside our minor arcs $\fm(X,Q)$, critically weakening our ``savings'' from the above lemma. To work around this issue, we reuse methods from \cite{daniels2023mobius}*{sec.\ 4}. 

First, we suppose that there exist quantities $K$, $\cX$, and $\cQ$ such that $k\xa \in \fm(\cX,\cQ)$ for $\xa \in \fm(X,Q)$ and $k \leq K$. Next, we observe that the definitions of the sets $\fM(\cX,\cQ)$ and $\fm(\cX,\cQ)$ imply that if $x\geq\cX$, then $\fm(\cX,\cQ) \subset \fm(x,\cQ)$. Using these two ideas together, we deduce that for all $x\geq\cX$, one has
    \<
        \label{eq:MinorSetup}
        \sup_{\substack{\xa \in \fm(X,Q) \\ k \leq K}}\lf| S_{f}(x, k\xa) \rh|\leq\sup_{\xq \in \fm(\cX,\cQ)} \lf| S_{f}(x, \xq) \rh|\leq\sup_{\xq \in \fm(x,\cQ)} | S_{f}(x,\xq) |.
    \>

We now consider how $K$, $\cX$, and $\cQ$ ought to be selected. Loosely speaking, we prioritize decreasing $\cX$ at the cost of increasing $\cQ$, demonstrating the mechanism of this ``exchange'' in the following lemma.

\begin{lemma}
\label{lem:MinorIncl}
	Fix $\xk \in (0,1)$, let $1 \leq Q \leq X$, let $\xa \in \fm(X,Q)$, and let $1 \leq k \leq Q^\xk$. If $(a_k,q_k)=1$ and 
		\[
			\lf|k\xa-\fr{a_k}{q_k}\rh|\leq\fr{Q^{1-\xk}}{q_kX},    
		\]
	then $q_k > Q^{1-\xk}$; in particular $\{k\xa\} \in \fm(X,Q^{1-\xk})$. Finally, if $0 < \nu < 1-\xk$ then 
		\[
			\fm(X,Q^{1-\xk}) \subset \fm(XQ^{-\nu},Q^{1-\xk-\nu}).
		\]
\end{lemma}

\begin{proof}
    The proof is an elementary application of Dirichlet approximation and the definitions of $\fM(X,Q)$ and $\fm(X,Q)$, as shown in \cite{daniels2023mobius}*{Lem.\ 4.6}.
\end{proof}

We now return to Lemma \ref{lem:MVCor0}; for convenience we provide a proof of said lemma using the proof outlined in \cite{montgomery1977exponential}, in which Lemma \ref{lem:MVCor0} is demonstrated as a corollary of the following theorem.

\begin{lemma}[{\cite{montgomery1977exponential}*{Thm.\ 1}}]
	\label{lem:MVThm}
	Suppose that $(a,q)=1$ and $q \leq x$. One has
		\[
			S_f(x,a/q) \less \fr{x}{\log 2x} + \fr{x}{\phi(q)^{1/2}} + \fr{x \log^{3/2}(2x/q)}{(x/q)^{1/2}}
		\]
	uniformly for all multiplicative $f$ with $|f| \leq1$, where $\xf(q)$ is the Euler totient function.
\end{lemma}

\begin{proof}[Proof of Lemma \ref{lem:MVCor0}]
	If $q=2$ or $R=2$ then \eqref{eq:MVCorBB0} is trivial, so assume that $q\geq3$, that $R\geq3$, and that $x\geq9$. For $r \in \nn$ with $r \leq x$ let
		\[
			\cR(x,r) := \fr{x}{\log 2x} + \fr{x}{\phi(r)^{1/2}} + \fr{x \log^{3/2}(2x/r)}{(x/r)^{1/2}}.
		\]
	Writing $e(\xa n) = e(\xb n) e((\xa-\xb)n)$ and summing by parts, one sees that
		\[
			S_f(x,\xa) = S_f(x,\xb) e((\xa-\xb)x) - 2\pi i (\xa-\xb) \int_1^x S_f(u,\xb) e((\xa-\xb)u) \,du.
		\]
	Suppose that $\xb = b/r$ with $(b,r)=1$ and $r \leq x$. Trivially bounding $S_f(u,\xb) \less u$ when $u \leq r$ and using Lemma \ref{lem:MVThm} when $r < u \leq x$, one finds that
		\[
			2\pi i (\xa-\xb) \int_1^x S_f(u,\xb) e((\xa-\xb)u) \,du \less |\xa - b/r|(r^2 + x \cR(x,r)),
		\]
	whereby
		\[
			S_f(x,\xa) \less \cR(x,r) + |\xa-b/r|(r^2 + x\cR(x,r)).
		\]
	
	If $q > x^{1/2}$, then setting $b = a$ and $r = q$ it follows that 
		\<
			\label{eq:MVCorBB2.5}
			S_f(x,\xa) \less \cR(x,r) + r^{-2}(r^2 + x\cR(x,r)) \less \cR(x,r).
		\>
	Now suppose that $q \leq x^{1/2}$. By Dirichlet's approximation theorem there exist coprime $b$ and $r$, with $r \leq 2x/q$, such that $|\xa-b/r| \leq q/(2xr)$. If $r = q$ then $|\xa - b/r|\leq 1/(2x)$ and we again recover \eqref{eq:MVCorBB2.5}. If $r \neq q$ then
		\[
			1\leq|ar - bq| = qr\lf|\fr{b}{r}-\fr{a}{q}\rh|\leq\fr{q^2}{2x} + \fr{r}{q} \leq \fr12 + \fr{r}{q},	
		\]
	whereby $r \geq q/2$ and $|\xa-b/r| \leq 1/x $, and it follows that $S_f(x,\xa) \less \cR(x,r) + x/q^{2}$.
	
	In each of the three scenarios for $x$, $q$, and $r$ described, we have $r\geq\tf12 q$ and
		\[
			S_f(x,\xa) \less \cR(x,r) + x / q^2.	
		\]
	Since $r \geq \fr{q}{2} \geq \fr{R}{2} \geq \tf32$ (so that $r\geq2$), we use the bound $\phi(n) \great \fr{n}{|\log\log{n}|}$ for $n\geq2$ (see, e.g., \cite{hardy2008introduction}*{Thm.\ 328}) to deduce that
		\[
			\phi(r) \great \fr{r}{|\log\log r|} \great \fr{R}{\log\log R},
		\] 
	and because $2 \leq R \leq q \leq2x/r$ it follows that
		\[
			\cR(x,r) \less \fr{x}{\log x} + \fr{x\log^{1/2}(\log R)}{R^{1/2}} + \fr{x \log^{3/2}(2x/r)}{(2x/r)^{1/2}} \less \fr{x}{\log x} + \fr{x \log^{3/2}R}{R^{1/2}},
		\]
	establishing Lemma \ref{lem:MVCor0}.
\end{proof}

As mentioned previously, ultimately we set $Q=X^{1/3}$, but for now we let $Q=X^\xd$ for some $\xd \in (0,\tfrac{1}{2})$ for illustrative purposes. We note that the upper bound of $\tfrac{1}{2}$ for $\xd$ ensures that distinct major arcs $\fM(a/q)$ do not overlap once $X$ is sufficiently large.

\begin{lemma}
	Let $\xd\in(0,\tfrac{1}{2})$, let $\nu,\kappa \in (0,1)$ satisfy $\nu + \kappa < 1$, let $Q=X^\xd$, and let $X$ be sufficiently large. For all $x\geq XQ^{-\nu}$ one has
		\<
			\label{eq:MinorSupBB}
			\sup_{\substack{a \in \fm(X,Q)\\k \leq Q^{\xk}}} |S_f(x,k\xa)| \less \fr{x}{\log{x}} + \fr{x\log^{3/2}R}{R^{1/2}},
		\>
	uniformly among multiplicative $f$ with $|f|\leq1$, where $R = Q^{1-\xk-\xn}$.
\end{lemma}

\begin{proof}
	It follows from Lemma \ref{lem:MinorIncl} that for $\xa \in \fm(X,Q)$ and $k \leq Q^\xk$, one has
		\<
			\label{eq:kxaMinor}
			\{k\xa\} \in \fm(XQ^{-\nu}, Q^{1-\xk-\nu}).    
		\>
	By Dirichlet's approximation theorem there exist $q_k$ and $a_k$ such that $(a_k,q_k)=1$, {\binoppenalty=10000\relpenalty=10000 $1 \leq q_k \leq XQ^{-\nu}/Q^{1-\xk-\nu}$}, and 
		\[
			\lf|\{k\xa\} - \fr{a_k}{q_k}\rh|\leq\fr{Q^{1-\xk-\xn}}{q_kXQ^{-\xn}}\leq\fr{1}{q_k^2}.	
		\]
	By \eqref{eq:kxaMinor} it holds that $q_k > Q^{1-\xk-\xn}$, and by the assumptions that $Q=X^\xd$ and that $X$ is sufficiently large, it additionally holds that
		\<
			\label{eq:MinorqkIneq}
			2 \leq Q^{1-\xk-\xn} \leq q_k\leq\fr{XQ^{-\xn}}{Q^{1-\xk-\xn}}. 
		\>
	As \eqref{eq:MinorqkIneq} remains valid with $XQ^{-\nu}$ replaced by any $x\geq XQ^{-\nu}$, setting $R:=Q^{1-\xk-\xn}$ and applying Lemma \ref{lem:MVCor0} one deduces that
		\[
			S_f(x,\{k\xa\}) \less \fr{x}{\log{x}} + \fr{x\log^{3/2}R}{R^{1/2}},
		\]
	and the result follows immediately since $S_f(x,\{k\xa\}) = S_f(x,k\xa)$.
\end{proof}

\begin{remark}
    Because the supremum in inequality \eqref{eq:MinorSupBB} is taken over all $k \leq Q^\xk$, it is important that the implicit constant in inequality \eqref{eq:MVCorBB0} does not depend on $q_k$ and thereby depend on $k$. This is why we emphasize said independence in Lemma \ref{lem:MVCor0}.
\end{remark}

Having established our desired bounds on $S_f(x,k\xa)$, we are ready to bound $\xY(\xr e(\xa))$ for $\xa \in \fm(X,Q)$ and, after taking $Q=X^{1/3}$, establish Theorem \ref{thm:Minor}. For convenience, we note that Theorem \ref{thm:Minor} is equivalent to the statement that: \emph{With $Q=X^{1/3}$, for all $X > X_0$ and $\xa \in \fm(X,Q)$ one has
    \[
			\xY(\xr e(\xa),f) \less \nfr{X}{\log X},
		\]
    uniformly among multiplicative $f$ with $|f|\leq1$.
}

\begin{proof}[Proof of Theorem \ref{thm:Minor}]
	Let $\xd \in (0,\tf12)$, let $\xk,\xn \in (0,1)$ be such that $\xk < \xn$ and $\xk+\xn<1$, let $Q = X^{\xd}$, let $X$ be sufficiently large, and let $k \leq Q^{\xk}$. Recall that $F_k(z) = \sum_{n=1}^\infty f^k(n) z^{nk}$. Summing by parts, we have
		\<
			\label{eq:MinorFInt}
			F_k(\xr e(\xa)) = \sum_{n=1}^\infty f^k(n)\xr^{nk} e(nk\xa) = \fr{k}{X} \int_0^\infty e^{-xk/X}S_f(x,k\xa) \,dx. 
		\>
	Let $R = Q^{1-\xk-\xn} = X^{\delta(1-\kappa-\nu)}$. Bounding $S_f(x,k\xa) \less x$ when $x \leq XQ^{-\nu}$ and using inequality \eqref{eq:MinorSupBB} when $x > XQ^{-\nu}$, we find that integral \eqref{eq:MinorFInt} is
		\begin{align}
			&\less \fr{k}{X} \lf[ \int_0^{XQ^{-\xn}} e^{-xk/X}x \,dx + \int_{XQ^{-\xn}}^\infty e^{-xk/X} \lf( \fr{x}{\logx} + \fr{x\log^{3/2}R}{R^{1/2}} \rh) \,dx \rh] \notag\\
			\label{eq:MinorFIntBB}
			&\less \fr{X}{k} \lf[ \int_0^{kQ^{-\xn}} e^{-u}u \,du + \lf( \fr{1}{\log XQ^{-\xn}} + \fr{\log^{3/2}R}{R^{1/2}} \rh) \int_{kQ^{-\xn}}^\infty e^{-u}u \,du \rh].
		\end{align}
	Observing that
		\[
			\fr{1}{\log XQ^{-\xn}} + \fr{\log^{3/2}R}{R^{1/2}} = \fr{1}{\log (X^{1-\xd \xn})} + \fr{\log^{3/2}(X^{\xd(1-\xk-\xn)})}{(X^{\xd(1-\xk-\xn)})^{1/2}} \less \fr{1}{\log X},	
		\]
	we integrate by parts in expression \eqref{eq:MinorFIntBB} 
	to deduce that
		\<
			\label{eq:MinorFIntBB2}
			F_k(\xr e(\xa)) \less \fr{X}{k} \lf[ 1-(1+kQ^{-\xn})e^{-kQ^{-\xn}} + \fr{1}{\log X} (1+kQ^{-\xn})e^{-kQ^{-\xn}} \rh].
		\>
	As $kQ^{-\xn} \leq Q^{\xk-\xn}$ and $\xk-\xn<0$, and since $Q$ is sufficiently large (because $X$ is), we use the approximation 
		\[
			(1+x)e^{-x} = 1 - \tf12 x^2 + O(x^3)
		\] 
	for small $x$ to bound the righthand side of \eqref{eq:MinorFIntBB2} as
		\begin{align*}
			&\less \fr{X}{k} \lf[ \tf12Q^{2\xk-2\xn} + O(Q^{3\xk-3\xn}) + \fr{1}{\logX} (1+O(Q^{2\xk-2\xn})) \rh] \less \fr{X/k}{\logX}. 
		\end{align*}
    From this, we conclude that
        \[
            F_k(\xr e(\xa)) \less \fr{X/k}{\logX} \qquad \text{for $\xa \in \fm(X,Q)$ and $k \leq X^{\xd \xk}$.}
        \]
    
    Summing over all $k$ and using Lemma \ref{F:lem:TailBB} then, we have
        \[
            \xY(\xr e(\xa)) = \lf[\sum_{k\leq X^{\xd \xk}} + \sum_{k>X^{\xd \xk}}\rh] \fr{F_k(\xr e(\xa))}{k} \less \fr{X}{\logX} \sum_{k \leq X^{\xd \xk}} \fr{1}{k^2} + X^{1-\xd \xk},
        \]
    and this latter expression is 
        \[
            \less \fr{X}{\log X} + \fr{X^{1-\xd \xk}}{\log{X}} + X^{1-\xd \xk} \less \fr{X}{\logX} + X^{1-\xd \xk}.
        \]
	Taking $\xd =\tf13$, $\xk=\tf13$, and $\xn = \tf25$ we see that
        \[
            \xk<\xn, \quad \xk+\xn<1, \quad\text{and}\quad X^{1-\xd \xk}=X^{\fr89},
        \] 
    and the result follows.
\end{proof}

\section{The major arcs for odd prime \tops{$p$}{p}}
\label{sec:major}

We now return to our specific consideration of $\xY(z)=\xY(z,\xc_p)$ for odd prime $p$. Specifically, we first derive first-order asymptotics for $\xY(\xr e(\xa))$, as $\xr\to1^{-}$, when $\xa$ is near a reduced rational $a/q \in [0,1]$. Using Mellin-transform methods one may compute lower order terms for $\xY(\xr e(\nfr{a}{q}))$, as seen in section \ref{sec:princ} and later in section \ref{sec:5}, but these computations become cumbersome as $q$ grows. At present we are simply interested in determining which major arcs $\fM(a/q)$ are the dominant contributors in our asymptotic formulae, and thus coarser methods suffice for now. 

As in the previous section, we begin with an undetermined $Q$ such that $1 \leq Q \leq X$. Now let $\xa = a/q+\xb$, with $q\leq Q$, $(a,q)=1$, and $|\xb|\leq Q/(qX)$, or in other words, let $\xa\in\fM(a/q)$ with $q\leq Q$, and recall from \eqref{eq:PsiF} that
	\[
		\xY(z) = \sum_{k=1}^\infty \fr 1 k \sum_{n=1}^\infty \xc^k(n) z^{nk} = \sum_{k=1}^\infty \fr{1}{k} F_k(z^k).
	\]
It is convenient to let
    \[
        \xt = X^{-1}(1-2\pi iX\xb),
    \]
so that for $\xa=\tf{a}{q}+\xb$ we may write $\xr e(\xa) = e(\tf{a}{q}) e^{-\xt}$. Doing this and letting $[p,q]:=\mathrm{lcm}[p,q]$, we have
	\[
		F_k(\xr^k e(k\xa)) = \sum_{n=1}^\infty \e{\fr{nka}{q}} \xc^{k}(n) e^{-nk \xt} 
		= \sum_{m = 1}^{[p,q]} e\lf(\fr{m k a}{q}\rh) \xc^{k}(m) \sum_{n \equiv  m \mod{[p,q]}} e^{-nk \xt}.
	\]
Since $[p,q] \ll q$ trivially, and
	\[
		m k |\xt| \less [p,q] k|\xt| \less q k \ABS{\fr{1}{X}-2\pi i\xb} \less qk \lf(\fr{1}{X} + \fr{Q}{qX}\rh) \less \fr{kQ}{X}
	\]
for $1 \leq m \leq [p,q]$, we have
	\[
		\sum_{n \equiv  m \mod{[p,q]}} e^{-nk\xt} = \fr{e^{-mk\xt}}{1 - e^{-[p,q]k\xt}} = \fr{1}{[p,q]k\xt}\pfr{1+O(mk\xt)}{1+O(qk\xt)} = \fr{1}{[p,q]k\xt} + \O{1}.
	\]
Again since $[p,q] \less q \less Q$, we deduce that
	\[
		F_k(\xr^k e(k\xa)) = \fr{1}{[p,q] k\xt} \sum_{m = 1}^{[p,q]} \e{\fr{m ka}{q}} \xc^{k}(m) + \O{Q}.
	\]

We are thus motivated to define, for $0 \leq a \leq q$, $(a,q)=1$, and $k \geq 1$, the quantities
    \<
        \label{eq:UkDefin}
        U_k(q,a) = \fr{1}{[p,q]} \sum_{m = 1}^{[p,q]} \e{\fr{m k a}{q}} \xc^k(m),
    \>
where again $[p,q]=\mathrm{lcm}(p,q)$, so that
	\[
		F_k(\xr^k e(k\xa)) = \fr{U_k(q,a)}{k\xt} + \O{Q}
	\]
and
	\<
		\label{eq:MajorKSum}
		\sum_{k\leq K} \fr{1}{k} F_k(\xr^k e(k\xa)) = \xt^{-1} \sum_{k\leq K} \fr{U_k(q,a)}{k^2} + \O{Q\log{K}}.
	\>
Because $|U_k(q,a)| \leq 1$, we complete the sum in \eqref{eq:MajorKSum} at the cost of $O(|\xt|^{-1}K^{-1}) = O(X/K)$ to deduce that
	\[
		\sum_{k\leq K} \fr{1}{k}F_k(\xr^k e(k\xa)) 
		= \xt^{-1} \sum_{k=1}^\infty \fr{U_k(q,a)}{k^2} + \O{\fr{X}{K}} + \O{Q\log{K}}.
	\]
Recalling from Lemma \ref{F:lem:TailBB} that $\sum_{k>K} k^{-1} F_k(z) \less X/K$ as well, we further deduce that
	\[
		\xY(\xr e(\xa)) = \xt^{-1} \sum_{k=1}^\infty \fr{U_k(q,a)}{k^2} + \O{\fr{X}{K}} + \O{Q\log{K}}.
	\]
Setting $Q = X^{1/3}$ and $K = X^{1/9}$ (consistent with the choices $\xd = \tf13$ and $\xk = \tf13$ from section \ref{sec:Minor}), we conclude that
    \<
        \label{eq:majorCrude}
        \xY(\xr e(\xa)) = \pfrac{X}{1-2\pi i X\xb} \sum_{k=1}^\infty \fr{U_k(q,a)}{k^2} + O(X^{\fr89}).
    \>

With formula \eqref{eq:majorCrude}, ``sorting'' the major arcs $\fM(a/q)$ based on their relative contributions to $\fp(n,\xc_p)$ is now a matter of computing $\sum_{k=1}^\infty U_k(q,a)/k^2$ for different $q$ and $a$. As we ultimately find that
    \[
        \sum_{k=1}^\infty \fr{U_k(1,1)}{k^2} = \sum_{k=1}^\infty \fr{U_k(2,1)}{k^2} = \fr{\pi^2}{24}\lf(1-\fr1p\rh),
    \]
which dictates the behavior of $\xY(\xr e(\xa))$ on our principal arcs, it is convenient in the remainder of this paper to let 
    \<
        \label{eq:cDef}
        \fc = \fr{\pi^2}{24}\lf(1-\fr1p\rh)
    \>
and normalize our constants using $\fc$. In particular, we define 
    \<
        \label{eq:VqDefin}
        V_q(a) = \frac{1}{\fc} \sum_{k=1}^\infty \fr{U_k(q,a)}{k^2} = \frac{1}{\fc[p,q]}\sum_{k=1}^\infty \sum_{m=1}^{[p,q]} \e{\fr{mka}{q}} \fr{\xc^k(m)}{k^2}
    \>
and note that $|V_q(a)| \leq 1$ for all $(a,q)=1$.

The explicit computations of different $V_q(a)$ are little more than exercises in the elementary properties of Dirichlet characters. Because $\xc^k$ is either $\xc_0 \mmod{p}$ or $\xc_p$ depending on the parity of $k$, and because the parity of $q$ also affects the formula for $V_q(a)$, the derivations of the different $V_q(a)$ are somewhat tedious. Thus, we skip these derivations presently and include them in the appendix. Granting this, we establish the following proposition.

\begin{proposition}
	\label{prop:Major}
	Let $\chi=\chi_p$, let $X>X_0$, and let $\xa = a/q + \xb$ with $(a,q) = 1$ and $|\xb| \leq 1/(qX^{2/3})$. There exists a constant $V_q(a)$ such that as $X\to\infty$ one has
		\<
			\label{eq:Major}
			\xY(\xr e(\xa)) = \frac{V_q(a)\fc X}{1-2\pi iX\xb} + O(X^{\fr89}), 
		\>
    where
        \[
            \fc = \fr{\pi^2}{24}\lf(1-\fr{1}p\rh)
        \] 
	and $V_q(a)$, as defined in \eqref{eq:VqDefin}, has the following closed forms. First, if $p \nmid q$ then 
			\<
				\label{eq:Vqcoprime}
				V_q(a) = \begin{cases}
					1/q^2 & 2\nmid q, \\
					4/q^2 & 2\mid q.
				\end{cases}
			\>
		If $p\mid q$ and $2\nmid q$, then 
			\<
                \label{eq:Vqprime}
				V_q(a) =
					\lf[ \xc(a)G(\xc)(1-\tfr{\xc(2)}{4})L_2(\xc)/\fc - 1 \rh]p/q^2,
			\>
        where $G(\xc)=\sum_{m=1}^p \xc(m)e(m/p)$ and $L_2(\chi) = L(2,\chi)$. Finally, if $p\mid q$ and $2\mid q$ then
            \<
                V_q(a) = -4p/q^2.
            \>
\end{proposition}

Because they feature prominently in our later discussions, we note that
		\[
			V_1(1) = V_2(1) = 1 \qquad\text{and}\qquad V_4(1) = V_4(3) = \tf14.
		\]

\section{The principal arcs}
\label{sec:princ}

Because $V_1(1)=V_2(1)$ as mentioned above, comparison of the contributions of the arcs near $0$, $1$, and $\tf12$ requires lower-order asymptotic terms for $\Psi(\rho e(\alpha),\chi_p)$. Thus, we now consider $\xY(\xr e(\xa))$ when $\xa$ is in one of the two principal arcs 
    \begin{align*}
        \fP &= \{\xa \in [0,1) : \|\xa\| \leq 3/(8\pi X) \}, \\
        \fP_* &= \{\xa \in [0,1) : |\xa-\tf12| \leq 3/(8\pi X_*) \}.
    \end{align*}
Complex $s$ are written $\xs+it$, and $\xs$ and $t$ always represent the real and imaginary parts of some $s\in\cc$, respectively. Integrals $\frac{1}{2\pi i} \int$ are abbreviated using dashed integrals $\dint$. For real $\xb$ we again set $\xt := X^{-1}(1-2\pi iX\xb)$ so that $z=\xr e(\xb)$ may be written as $e^{-\xt}$.

Because $\Psi(\xr e(\xa)) = \Psi(\xr e(\xa-1))$ for all $\xr$ and $\xa$, it is convenient to identify the sets
    \<
        \label{eq:prcId}
        \{ e(\xa) : \xa \in \fP \} \qquad\text{and}\qquad \{ e(\xb) : \text{$\xb\in\rr$ and $|\xb| \leq  3/(8\pi X)$} \}
    \>
and consider $\xY(\xr e(\xa))$ with $\xa \in \fP$ as $\xY(\xr e(\xb))$ with $|\xb| \leq 3/(8\pi X)$. As $e(1/2)=-1$, we may similarly identify the sets
    \<
        \label{eq:prc*Id}
        \{ e(\xa) : \xa \in \fP_* \} \qquad\text{and}\qquad \{ -e(\xb) : \text{$\xb\in\rr$ and $|\xb| \leq 3/(8\pi X_*)$} \}   
    \>
and consider $\xY(\xr_* e(\xa))$ with $\xa\in\fP_*$ as $\xY(-\xr_* e(\xb))$ with $|\xb| \leq 3/(8\pi X_*)$.

\begin{remark}
    Despite the above definitions' in terms of $X_*$, we do not presently distinguish $X$ and $X_*$ because, at present, both $X$ and $X_*$ are simply large real parameters. Thus, in light of \eqref{eq:prcId} and \eqref{eq:prc*Id}, we see that our results on $\Psi(\pm \xr e(\xb))$ with $|\xb| \leq 3/(8\pi X)$ sufficiently cover the analyses of $\Psi(\xr e(\xa))$ for $\xa$ in the principal arcs $\fP$ and $\fP_*$.
\end{remark}

Observing that $\xc^k(n) = \xc(n)$ for odd $k$ and $\xc^k(n) = \xc_0(n)$ for even $k$, where $\xc_0$ is the principal character (mod $p$), we split $\xY(z)$ into two sums 
    \<
        \label{eq:Psi0Psi1Def}
        \xY(z) = \sum_{k \equiv  0\mod2} \sum_{n=1}^\infty \frac{\xc_0(n)}{k} z^{nk} + \sum_{k\equiv  1\mod2} \sum_{n=1}^\infty \frac{\xc(n)}{k} z^{nk},
    \>
which we denote by $\xY_0(z)$ and $\xY_1(z)$ in the obvious manner, so that
    \[
        \xY(z) = \xY_0(z) + \xY_1(z).
    \]
Using the well-known formula $e^{-w} = \ldint{\xs} \xG(s)w^{-s} \,ds$ for $\Re w > 0$ and $\xs > 0$, we express $\xY_0(z)$ and $\xY_1(z)$ as sums of integrals, interchange the summations and integrations using the series' absolute convergence, and thereby derive the formulae
	\begin{align}
        \label{eq:Psi0Mellin}
	    \xY_0(\xr e(\xb)) = \xY_0(e^{-\xt}) &= \ldint{2} 2^{-s-1}\xz(s+1) L(s,\xc_0) \xG(s)\xt^{-s} \,ds, \\
        \label{eq:Psi1Mellin}
	    \xY_1(\xr e(\xb)) = \xY_1(e^{-\xt}) &= \ldint{2} (1-2^{-s-1}) \xz(s+1) L(s,\xc) \xG(s)\xt^{-s} \,ds.
    \end{align}
In what follows we again use the conventions that
    \[
        L_r(\xc) = L(r,\xc) \qquad\text{and}\qquad L_r'(\xc) = L'(r,\xc) \qquad (r=0,1,2).
    \]

\begin{lemma}
	\label{lem:prcPsi}
	For all $|\xb| \leq 3/(8\pi X)$, as $X \to \infty$ one has
		\<
			\label{eq:prcPsi}
            \xY(\xr e(\xb)) = \fr{\fc X}{1-2\pi iX\xb} + 2\xl \log\pfrac{X}{1-2\pi iX\xb} + \xL + O(X^{-\fr12}),
		\>
	where
        \[
            \fc = \tf{\pi^2}{24}(1-\tf1p), \quad \xl = \tf14 L_0(\xc), \quad\text{and}\quad \xL = \tf12\big[L_0'(\xc) + L_0(\xc)\log{2}\big] - \tf14 \log{p}.
        \]
\end{lemma}

\begin{proof}
	As $L(s,\xc_0) = (1-p^{-s})\xz(s)$, equation \eqref{eq:Psi0Mellin} states that
		\<
			\label{eq:Psi0Int}
			\xY_0(\xr e(\xb)) = \ldint{2} 2^{-s-1}(1-p^{-s})\xz(s+1)\xz(s)\xG(s)\xt^{-s} \,ds,
		\>
    where again we write $\dint$ to indicate $\tfrac{1}{2\pi i}\int$.
	Letting $T$ be a large parameter to be chosen below, we first truncate the integral \eqref{eq:Psi0Int} at height $T$ with error term
        \<
            \label{eq:Psi0TruncError}
            \less \int_T^\infty \ABS{\xG(2+it)\xt^{-2-it}} \,dt.
        \>
    Having assumed that $|\xb| \leq 3/(8\pi X)$ we have
		\[
			\lf|\arg \xt \rh| = \lf| \arg(1-2\pi iX\xb) \rh| \leq \lf|\arctan(\tf34)\rh| \leq \tf{\pi}{4},
		\]
	which implies that
		\<
            \label{eq:tauAbsBB}
			\lf|\xt^{2-it}\rh| = |\xt|^{-2} e^{t \arg \xt} \leq X^2 e^{\pi|t|/4}. 
		\>
    We now recall two bounds on $\xG(s)$ and $\xz(s)$. First, one has \cite{paris2001asymptotics}*{ineq.\ (2.1.19)}
        \<
            \label{eq:GammaBB}
            |\xG(s)| \less |s|^{\xs - \fr12} \expp{-\tf12 \pi|t| + \tf16 |s|^{-1}} \qquad (\xs \geq 0).
        \>
    Second, in any halfplane $\xs\geq\xs_0$, as $|t| \to \infty$ one has \cite{titchmarsh1986theory}*{sec.\ 5.1}
        \<
            \label{eq:ZetaHalfplaneBB}
            |\xz(s)| \less |t|^{k} \qquad \text{for some $k = k(\xs_0)$.}
        \>

	Using inequalities \eqref{eq:tauAbsBB}, \eqref{eq:GammaBB}, and \eqref{eq:ZetaHalfplaneBB}, it follows that integral \eqref{eq:Psi0TruncError} is
		\[
			\less X^2 \int_T^\infty t^{c_1} e^{-\pi t/4} \,dt \less X^2 e^{-\pi T/5},
        \]
    whereby
        \<
            \label{eq:Psi0Trunc}
            \xY_0(e^{-\xt}) = \dint_{2-iT}^{2+iT} 2^{-s-1}(1-p^{-s})\xz(s+1)\xz(s)\xG(s)\xt^{-s} \,ds + O(X^2 e^{-\pi T/5}).
        \>
	The integrand in integral \eqref{eq:Psi0Trunc} is analytic in the halfplane $\xs \geq -\tf12$, with the exceptions of simple poles at $s=1$ and $s=0$; one may compute their residues to be
		\[
			\tf{\pi^2}{24} (1-\tf{1}{p})\xt^{-1} = \fc \xt^{-1} \qquad\text{and}\qquad -\tf14 \log p,	
		\]
	respectively. With these and \eqref{eq:GammaBB} and \eqref{eq:ZetaHalfplaneBB}, it is straightforward to deduce that
		\[
			\xY_0(e^{-\xt}) = \fc \xt^{-1} - \tf14\log p + O(X^2 e^{-\pi T/5}) + O(X^{-1/2}).
        \]
    Selecting $T = X$ and replacing $\xt^{-1}$ with $X^{-1}(1-2\pi iX\xb)$, we conclude that
        \<
            \label{eq:Psi0form}
			\xY_0(\xr e(\xb)) = \fr{\fc X}{1-2\pi iX\xb} - \tf14 \log{p} + O(X^{-\fr12}).
		\>

    Now for $\xY_1(\xr e(\xb))$, equation \eqref{eq:Psi1Mellin} states that
        \[
            \xY_1(\xr e(\xb)) = \ldint{2} (1-2^{-s-1})\xz(s+1)L(s,\xc)\xG(s)\xt^{-s} \,ds.
        \]
    The integrand here is analytic when $\xs\geq-\tf12$ except at $s=0$, where it has a double pole with residue
        \[
            -\tf12 L_0(\xc) \log{\xt} + \tf12\big[L_0'(\xc) + L_0(\xc)\log{2}\big].
        \]
    Using arguments similar to those for \eqref{eq:Psi0form}, one may check that
        \[
            \xY_1(\xr e(\xb)) = -\tf12 L_0(\xc) \log{\xt} + \tf12\big[L_0'(\xc) + L_0(\xc)\log{2}\big] + O(X^{-\fr12}),
        \]
    and combining this with \eqref{eq:Psi0form} yields formula \eqref{eq:prcPsi}.
\end{proof}

When $z = -\rho e(\xb) = -e^{-\xt}$ our formulae require only small changes. As $(-z)^{nk} = z^{nk}$ for even $k$, we have $\xY_0(-z)=\xY_0(z)$, and thus by \eqref{eq:Psi0form} we immediately conclude that
	\[
		\xY_0(-\xr e(\xb)) = \fr{\fc X}{1-2\pi iX\xb} - \tf14 \log{p} + O(X^{-\fr12}).
	\]
As $(-1)^{nk} = (-1)^n$ for odd $k$, we have
    \[
        \xY_1(-z) = \sum_{k\equiv  1\mod2} \sum_{n=1}^\infty (-1)^n k^{-1} \xc(n) z^{nk}, 
    \]
using which we derive a formula like \eqref{eq:Psi1Mellin} by replacing $L(s,\xc)$ with a suitable closed form for $\sum_{n=1}^\infty (-1)^n \xc(n)n^{-s}$. In particular, we may at once verify that 
	\<
        \label{eq:DirChi-1}
		\sum_{n=1}^\infty (-1)^n\xc(n)n^{-s} = (\xc(2)2^{1-s} - 1)L(s,\xc) \qquad (\Re(s)>1),
	\>
and using this in equation \eqref{eq:Psi1Mellin} we have
	\[
		\xY_1(-\xr e(\xb)) = \ldint{2} (1-2^{-s-1})(2^{1-s}\xc(2)-1)\xz(s+1)L(s,\xc)\xG(s)\xt^{-s} \,ds.
	\]
The integrand here has a double pole at $s=0$, this time with residue
	\begin{align*}
		& -\tf12(2\xc(2)-1) L_0(\xc)\log{\xt} + \tf12\big[(2\xc(2)-1)L'_0(\xc) - L_0(\xc)\log{2}\big].
	\end{align*}
Arguing as done for Lemma \ref{lem:prcPsi}, we easily establish the following lemma.

\begin{lemma}
	\label{lem:prcPsi*}
	For all $|\xb| \leq 3/(8\pi X)$, as $X \to \infty$ one has
		\[
            \xY(-\xr e(\xb)) =  \fr{\fc X}{1-2\pi iX\xb} + 2\xl_*\log\pfrac{X}{1-2\pi iX\xb} + \xL_* + O(X^{-\fr12}),
		\]
    where
        \[
            \fc=\tf{\pi^2}{24}(1-\tf1p), \qquad \xl_* = \tf14(2\xc(2)-1)L_0(\xc),
        \]
    and
        \[
            \xL_* = \tf12\big[(2\xc(2)-1)L'_0(\xc) - L_0(\xc)\log{2}\big] - \tf14\log{p}.
        \]
\end{lemma}

\section{The relationships between \tops{$n$}{n}, \tops{$X$}{X}, and \tops{$X_*$}{X*}}
\label{sec:relations}

Recalling that $\xF(z)=\exp \xY(z)$, $\xr=\exp(-1/X)$, and $\xr_*=\exp(-1/X_*)$, we now consider how the radius of integration $\xr$ in 
	\[
        \label{eq:pncIntRecall}
		\dint_{|z|=\xr} \xF(z)z^{-n-1}\,dz = \xr^{-n} \int_0^1 \Phi(\xr e(\xa))e(-n\xa) \,d\xa = \xr^{-n} \int_0^1 \exp\lf(\xY(\xr e(\xa)) - 2\pi in\xa\rh) \,d\xa
	\]
should depend on $n$. Having the advance knowledge that this integral is dominated by the contributions from those $\xa$ close to $0$, $\fr12$, and $1$, i.e., those $\xa$ in $\fP$ and $\fP_*$, we pick $\xr = \xr(n)$ and $\xr_* = \xr_*(n)$ so as to individually minimize 
	\begin{align*}
		\xr^{-n} \exp \xY(\xr) &= \expb{\xY(\xr)+\fr{n}{X}} \qquad\text{and}\qquad\xr_*^{-n} \exp \xY(-\xr_*) = \expb{\xY(-\xr_*)+\fr{n}{X_*}}
	\end{align*}
as functions of $n$. We first focus on minimizing $\exp(\xY(\xr)+n/X)$.

It follows from Lemma \ref{lem:prcPsi} that
    \<
		\label{eq:Psirho}
		\xY(\xr) = \fc X + 2\xl\logX + \xL + O(X^{-1/2}) \qquad (X\to\infty),
    \>
where $\fc = \tf{\pi^2}{24}(1-\tf1p)$ and $\xl$ and $\xL$ are given in said lemma (though their particular values are not relevant here), and thus 
	\<
        \label{eq:Expbexpr}
		\expb{\xY(\xr)+\fr{n}{X}} = e^{\xL} \expb{ \fr{n}{X} + \fc X + 2\xl\logX + O(X^{-\fr12})}.
	\>
Since the $O(X^{-1/2})$ term in \eqref{eq:Expbexpr} is $o(1)$ as $X\to\infty$, it is reasonable to minimize $n/X + \fc X + 2\xl\logX$ as a function of $X$ there. This is minimized when $\fc X^2+2\xl X-n=0$, i.e., when
	\<
		\label{eq:XExact}
		X = \big(\rt{\fc n+\xl^2}-\xl\big)/\fc.
	\>
We thus define $X = X(n)$ by \eqref{eq:XExact} and note that
	\<
		\label{eq:XnAsymp}
		n/X = \fc X + 2\xl \qquad\text{and}\qquad X = (n/\fc)^{1/2} - \xl/\fc + O(n^{-1/2}),
	\>
where the latter relation holds as $n\to\infty$. 

As Lemma \ref{lem:prcPsi*} provides a formula for $\xY(-\xr_*)$ similar to \eqref{eq:Psirho}, namely the formula
    \[ 
        \xY(-\xr_*) = \fc X_* + 2\xl_*\log{X_*} + \xL_* + O(X_*^{-1/2}) \qquad (X_*\to\infty),
    \]
we are similarly led to define $X_*$ via
	\<
		\label{eq:X*Exact}
		X_* = \big(\rt{\fc n+\xl_*^2}-\xl_*\big)/\fc.
	\>
Doing so, we similarly have
	\<
		\label{eq:X*nAsymp}
		\nfr{n}{X_*} = \fc X_* + 2\xl_* \qquad\text{and}\qquad X_* = (n/\fc)^{1/2} - \nfr{\xl_*}{\fc} + O(n^{-1/2}),
	\>
the latter statement again holding as $n\to\infty$.

\section{Arc transference and \tops{$*$}{*}-arcs}
\label{sec:transference}

We now recall the ``arc transference'' of \cite{daniels2023mobius}*{sec.\ 9} used to modify the contour of integration for
    \<
        \label{eq:ints}
        \dint_{|z|=r} \Phi(z)z^{-n-1} \,dz = r^{-n} \int_0^1 \exp \xY(re(\xa)) e(-n\xa) \,d\xa,
    \>
where again the dashed integral $\dint$ abbreviates $\fr{1}{2\pi i} \int$. In this integral, control of our error terms when $\xa$ is near $0$ and $1$ requires that the radius $r$ be set to $\xr = \exp(-1/X)$, with $X$ defined as in \eqref{eq:XExact}. On the other hand, control of our error terms in \eqref{eq:ints} for $\xa$ near $\tf12$ similarly requires that $r = \xr_* = \exp(-1/X_*)$, with $X_*$ is defined as in \eqref{eq:X*Exact}. Despite the similar definitions of $X$ and $X_*$ in equations \eqref{eq:XExact} and \eqref{eq:X*Exact}, respectively, we cannot expect that $X$ and $X_*$ are equal in general. Indeed, in section \ref{sec:AsymptoticsMod8} we see that $X\neq X_*$ when $p\equiv 3\mmod{8}$. Thus, a compromise is needed.

As $\Phi(z)$ is analytic for $|z|<1$, we split the interval $[0,1)$ into the sets 
    \[
        \fA := [0,\tf14)\cup[\tf34,1) \qquad\text{and}\qquad \fA_* := [\tf14,\tf34),
    \]
and change the path of integration on the left side of \eqref{eq:ints} to be the positively oriented path consisting of the two semicircles 
    \[
        \xr e(\fA) := \{ \xr e(\xa) : \xa \in \fA \} \qquad\text{and}\qquad \xr_* e(\fA_*) := \{ \xr_* e(\xa) : \xa \in \fA_* \}
    \]
and the two connecting line segments $[i\xr,i\xr_*]$ and $[-i\xr_*,-i\xr]$. Using this new contour in \eqref{eq:ints}, we have
    \<
        \label{eq:pxcNewInt}
        \fp(n,\xc) = \dint_{\xr e(\fA)} \Phi(z) z^{-n-1}\,dz + \dint_{\xr_* e(\fA_*)} \Phi(z) z^{-n-1} \, dz + \fT(\xr,\xr_*),
    \>
where
    \[
        \fT(\xr,\xr_*) := \dint_{i\xr}^{i\xr_*} \Phi(z)z^{-n-1} \,dz + \dint_{-i\xr_*}^{-i\xr} \Phi(z)z^{-n-1} \,dz.
    \]

We now show that the transference term $\fT(\xr,\xr_*)$ is small compared to the main integrals in \eqref{eq:pxcNewInt}, both of which grow like $n^{c_0}\exp(2\rt{\fc n})$, with $\fc=\fr{\pi^2}{24}(1-\tfr{1}{p})$ and some constant $c_0$, as shown in sections \ref{sec:MainInts} and \ref{sec:Asymptotics}.

\begin{lemma}
    \label{lem:Transfer}
    Let $\fc=\fr{\pi^2}{24}(1-\tfr{1}{p})$, let $\xl$ and $\xl_*$ be real constants, and for $n > n_0$ let 
        \<
            \label{eq:XX*Transfer}
            X = \big(\rt{\fc n+\xl^2}-\xl\big)/\fc \qquad\text{and}\qquad X_* = \big(\rt{\fc n+\xl_*^2}-\xl_*\big)/\fc.
        \>
    As $n\to\infty$, one has 
        \[
            \fT(\xr,\xr_*) \less \exp(\tf32\rt{\fc n}).
        \]
\end{lemma}

\begin{proof}
    Fix $\xe>0$ and let $n>n_0$. It is evident that $\fT(\xr,\xr_*) = -\fT(\xr_*,\xr)$, so there is no loss of generality in assuming that $X \leq X_*$, and therefore that $\xr \leq \xr_*$. We first note, by \eqref{eq:XX*Transfer}, that 
        \[
            n/X = \fc X + 2\xl \qquad\text{and}\qquad n/X_* = \fc X_* + 2\xl_*,
        \]
    and further note that as $n\to\infty$ we have
		\begin{align*}
			\fc X+2\xl &= \rt{\fc n+\xl^2} + \xl = \rt{\fc n} + O(n^{-1/2}), \\
			\fc X_*+2\xl_* &= \rt{\fc n+\xl_*^2} + \xl_* = \rt{\fc n} + O(n^{-1/2}).	
		\end{align*}
    Next, observing that $\xr^{-n} = e^{n/X}$ and $\xr_*^{-n} = e^{n/X_*}$, it follows that
        \begin{align}
            \label{eq:TIntBB}
            \begin{aligned}
            \int_{\xr}^{\xr_*} \frac{dt}{t^{n+1}} = \fr{1}{n}(\xr^{-n}-\xr_*^{-n}) = \fr{1}{n}(e^{\fc X+2\xl} - e^{\fc X_*+2\xl_*}) \less e^{\rt{\fc n}}.
            \end{aligned}
        \end{align}
	
    Now let $t = \exp(-1/X_t)$ for $\xr \leq t \leq \xr_*$. Recalling that $V_4(1) = V_4(3) = \fr14$ from the remark following Proposition \ref{prop:Major}, we have
        \<
            \label{eq:PsiitBB}
            \Re \xY(it) = \tf14 \fc X_t + O(X_t^{8/9}) \leq \tf14 \fc X_* + O(X_*^{8/9}) = \tf14 \rt{\fc n} + O(n^{4/9})
        \>
	for $\xr \leq t \leq \xr_*$, which, for $n > n_0$, is less than $(\tf14+\xe)\rt{\fc n}$. Using inequalities \eqref{eq:TIntBB} and \eqref{eq:PsiitBB}, we find that 
        \[
            \lf| \dint_{i\xr}^{i\xr_*} \Phi(z)z^{-n-1} \,dz \rh| 
            \less \exp((\tf14+\xe)\rt{\fc n}) \int_{\xr}^{\xr_*}\frac{dt}{t^{n+1}} \less \exp((\tf54+\xe)\rt{\fc n}),
        \]
    and the result follows by taking $\xe$ to be sufficiently small.
\end{proof}

\begin{corollary}
    \label{cor:Transfer}
    For $n > n_0$ let $\xr=e^{-1/X}$ and $\xr_*=e^{-1/X_*}$, with $X$ and $X_*$ as defined in Lemma \ref{lem:Transfer}. As $n\to\infty$ one has
        \[
            \begin{aligned}
                \fp(n,\xc_p) 
                &= \xr^{-n} \int_{\fA} \Phi(\xr e(\xa)) e(-n\xa) \, d\xa + \xr_*^{-n} \int_{\fA_*} \Phi(\xr_* e(\xa)) e(-n\xa)\,d\xa + O\big( e^{\fr32\sqrt{\fc n}} \big).
            \end{aligned}
        \]
\end{corollary}

\section{The contributions of the nonprincipal arcs}
\label{sec:Nonprincipal}

We now show, for odd primes $p\neq 5$, that the magnitude of $\xF(\xr e(\xa))$ is largest when $\xa$ is close to $0$, $1$, or $\fr12$. As $|\xF(z)| = \exp(\Re \xY(z))$, we accomplish this by using the formulae of Proposition \ref{prop:Major} to compare $\Re \xY(\xr e(\xa))$ for different $\xa$. We again recall that
    \begin{align*}
        \fP &= \{\xa \in [0,1) : \|\xa\| \leq 3/(8\pi X) \}, \\
        \fP_* &= \{\xa \in [0,1) : |\xa-\tf12| \leq 3/(8\pi X_*) \}.
    \end{align*}

\begin{lemma}
    \label{lem:RePsiBB}
	Let $p\neq 5$ be an odd prime and let $X$ and $X_*$ be sufficiently large. There exists a constant $0<\xd<1$ such that: for all $\xa \in \fA\setminus\fP$ one has
		\<
			\label{eq:PsiReBB}
			\Re \xY(\xr e(\xa)) \leq \xd \xY(\xr),
		\>
	and for all $\xa \in \fA_*\setminus\fP_*$ one has
		\<
			\label{eq:PsiRe*BB}
			\Re \xY(\xr_*e(\xa)) \leq \xd \xY(-\xr_*). 
		\>
\end{lemma}

\begin{proof}
    Letting $Q = X^{1/3}$, we suppose that $\xa\in\fA$ can be written as $\xa = a/q+\xb$ with $0 \leq a \leq q \leq X^{1/3}$, $(a,q)=1$, and $|\xb| \leq 1/(qX^{2/3})$. Defining 
        \[
            \xD = (1+4\pi^2X^2\xb^2)^{-\nfr12} \qquad\text{and}\qquad \varphi = \arctan(2\pi X\xb)
        \]
    with $\arctan(u) \in (-\tf{\pi}{2},\tf{\pi}{2})$, equation \eqref{eq:Major} states that
        \[
            \xY(\xr e(\xa)) = V_q(a)\fc X\xD e^{i\varphi} + O(X^{\nfr89}).
        \]
    Recalling that $V_1(1)= 1$ (which we reconfirmed in Lemma \ref{lem:prcPsi}) we have
        \[
            \xY(\xr) = \fc X + O(X^{\nfr89}),
        \]
    and thus to establish inequality \eqref{eq:PsiReBB} it suffices to show that 
        \<
            \label{eq:ReBB}
            \Re(V_q(a)\xD e^{i\varphi}) \leq \xd \qquad \text{for some $\xd < 1$.}
        \>

    If $q>2$ and $(p,q)=1$, then $V_q(a) \leq \fr14$ by \eqref{eq:Vqcoprime}. As $V_q(a) < 0$ when $2p\mid q$, inequality \eqref{eq:ReBB} holds \emph{a fortiori} as $X\to\infty$. It thus remains to consider the case where $p\mid q$ but $2\nmid q$, say $q=(2r-1)p$ with $r \in \nn$. Here
        \<
            \label{eq:VqDele}
            V_q(a)\xD e^{i\varphi} = \lf[ \xc(a)G(\xc)\lf(1-\fr{\xc(2)}{4}\rh)\fr{L_2(\xc)}{\fc} - 1 \rh]\frac{\xD e^{i\varphi}}{(2r-1)^2p},
        \>
    where again $L_2(\chi)=L(2,\chi)$. Since we are interested in upper bounds for the real part of \eqref{eq:VqDele}, it suffices to let $r=1$. 
    
    To handle the cumbersome term in square brackets in \eqref{eq:VqDele}, it is convenient to let
        \[
            W_p 
            = \lf(1-\fr{\xc(2)}{4}\rh)\fr{L_2(\xc)}{\fc} 
            = \fr{6(4-\xc(2))L_2(\xc)}{\pi^2(1-1/p)} \qquad (\text{since $\fc = \tf{\pi^2}{24}(1-\tf1p)$})
        \]
    and record the following crude bounds on $W_p$. Namely, because
        \[
            L_2(\xc) = \prod_{\substack{\l\text{ prime}\\ \l\neq p}} \Big(\fr{1}{1-\xc(\l)\l^{-2}}\Big) \geq \prod_{\substack{\l\text{ prime}\\ \l\neq p}} \Big(\fr{1}{1+\l^{-2}}\Big) = (1+p^{-2})\fr{\xz(4)}{\xz(2)} = (1+p^{-2})\fr{\pi^2}{15}
        \]
    and
        \[
            L_2(\xc) = \sum_{k=1}^\infty \fr{\xc(k)}{k^2}\leq\sum_{k=1}^\infty \fr{\xc_0(k)}{k^2} = (1-p^{-2})\fr{\pi^2}{6},
        \]
    after simplifying we deduce that
        \<
            \label{eq:WBB}
            1 < \tf25 (4-\xc(2))\pfr{1+p^{-2}}{1-p^{-1}} \leq W_p \leq (4-\xc(2))\left(1+\frac{1}{p}\right).
        \>

    Next we recall the classical facts that the Gauss sum $G(\xc_p)=\rt p$ if $p\equiv 1\mmod4$, and $G(\xc_p)=i\rt p$ if $p\equiv 3\mmod4$. 
    
    Now suppose that $p \equiv  1 \mmod{4}$. As $\xD \leq 1$, it follows from \eqref{eq:VqDele} and \eqref{eq:WBB} that
        \[
            \Re(V_q(a)\xD e^{i\varphi}) \leq (\rt{p}W_p - 1)/p,
        \]
    and thus it suffices to show that $(\rt{p}W_p - 1)/p < 1$. Numerical estimates show that, among primes $p > 5$ such that $p\equiv 1 \mmod{4}$, one has
        \[
            \fr{\rt{p}W_p-1}{p} \leq W_{13} = \fr{1}{13} \bigg( \fr{65\sqrt{13}}{2\pi^2}L_2(\chi_{13}) - 1 \bigg) \approx 0.69,
        \]
    which validates inequality \eqref{eq:ReBB} when $p \equiv  1 \mmod{4}$ (and $p \neq 5$) and $\xa \in \fA \cap \fM(\tf{a}{(2r-1)p})$. 

    We remark that since $L_2(\xc_5) = \tf{4\pi^2\rt{5}}{125}$ (see, e.g., Lemma \ref{lem:MV10.1.14}), one has
        \[
            \tf15(\rt{5}W_5 - 1) = 1,	
        \]
    which implies that $V_5(1)=V_5(4)=1$. Thus $\xY(\xr e(\fr15))$ and $\xY(\xr e(\tf45))$ are both ``on par'' with $\xY(\xr)$ and $\xY(-\xr_*)$ as $\xr$ and $\xr_*$ tend to $1$. This is precisely why we specially treat $\xY(z,\xc_5)$ and $\fp(n,\xc_5)$ in section \ref{sec:5}.

    Returning to \eqref{eq:VqDele}, suppose now that $p\equiv 3\mmod4$. Now $G(\xc)=i\rt p$, whereby
        \[
            V_q(a)\xD e^{i\varphi} = \fr1p \big( i\xc(a)\rt{p}W_p - 1 \big)\xD e^{i\varphi}.
        \]
    Since 
        \[
            e^{i\varphi} = \fr{1+ i(2\pi X\xb)}{(1+4\pi^2X^2\xb^2)^{1/2}} = \xD(1+i(2\pi X\xb))
        \] 
    and $\xD^2 = (1+(2\pi X\xb)^2)^{-1}$, we have
        \<
            \label{eq:RePart}
            \Re(V_q(a)\xD e^{i\varphi}) = \fr{-1-(2\pi X\xb)\xc(a)\rt{p}W_p}{(1+(2\pi X\xb)^2)p}.
        \>
    As we are interested in bounding \eqref{eq:RePart} above by some $\xd<1$, it suffices to let $\xc(a)=-1$ and $\xb>0$, set $u:=(2\pi X\xb)$, and show that
        \[
            \fr{-1+u\rt{p}W_p}{(1+u^2)p} < 1, \qquad\text{i.e., that}\qquad u^2 - (W_p/\rt{p})u + (1+1/p) > 0,
        \]
    for \emph{all} $u \in \rr$. This holds if and only if $(1+\nfr{1}{p}) > W_p^2/(4p)$, or equivalently if and only if $W_p^2/(4(p+1))<1$. Using the crude bound \eqref{eq:WBB} we have
        \<
            \label{eq:p34CB}
            \fr{W^2_p}{4(p+1)} \leq \fr{(4-\xc(2))^2(1+\nfr1p)^2}{4(p+1)}, 
        \>
    and for $p \geq 7$ the right side of \eqref{eq:p34CB} is maximized when $p = 11$ (regardless of the residue of $p$ modulo 4), where it has value $\tf{300}{484} \approx 0.62$. Numerically computing $W_3^2/16$ we have
        \[
            \tf{1}{16}W_3^2 = \tfr{1}{16}\lf(\tf{45}{\pi^2}L_2(\chi_3)\rh)^2 \approx 0.79, 	
        \]
    and therefore inequality \eqref{eq:ReBB} holds when $p \equiv  3 \mmod{4}$ and $\xa \in \fA \cap \fM(\tf{a}{(2r-1)p})$.

    In total now, we have established that inequality \eqref{eq:PsiReBB} holds for $\xa \in \fA\cap\fM(X,X^{1/3}) \setminus (\mathfrak{M}(\frac{0}{1})\cup\mathfrak{M}(\frac{1}{1}))$. Since Theorem \ref{thm:Minor} establishes that
        \[
            \xY(\xr e(\alpha)) \ll X/\log{X} \qquad \text{for $\alpha \in \fA \cap \fm(X,X^{1/3})$},
        \]
    inequality \eqref{eq:PsiReBB} holds \emph{a-fortiori} for those $\alpha$ as well. Thus, it remains to handle those $\xa \in (\fM(\fr{0}{1}) \cup \fM(\fr11))\setminus \fP$. However, if $\xa=\xb \in \fM(\fr01)$ or $\xa=1-\xb \in \fM(\fr11)$ with $\xb \geq 3/(8\pi X)$, then inequality \eqref{eq:ReBB} is immediate since
        \[
            \xD \leq (1 + \tf{36}{64})^{-\fr12} < 1,
        \]
    completing the proof of inequality \eqref{eq:PsiReBB}. The arguments for \eqref{eq:PsiRe*BB} for $\xa \in \fA_* \setminus \fP_*$ and \eqref{eq:PsiRe*BB} are identical, mutatis mutandis.
\end{proof}

Recalling the result of Corollary \ref{cor:Transfer}, we now use Lemma \ref{lem:RePsiBB} to further isolate the main terms in our asymptotic formula for $\fp(n,\xc_p)$, which we recall are ultimately of order $n^{c_0}\exp(2\rt{\fc n})$ for some $c_0$.

\begin{proposition}
	\label{prop:pncReduc}
	Let $p\neq 5$ be an odd prime, let $\fc=\tf{\pi^2}{24}(1-\tf{1}{p})$, let $\xl$ and $\xl_*$ be the real constants from Lemmata \ref{lem:prcPsi} and \ref{lem:prcPsi*}, respectively, and for $n > n_0$ let
        \<
            \label{eq:XX*ExactRecall}
            X = \big(\rt{\fc n+\xl^2}-\xl\big)/\fc \qquad\text{and}\qquad X_* = \big(\rt{\fc n+\xl_*^2}-\xl_*\big)/\fc.
        \>
    For some constant $\xd > 0$, as $n\to\infty$ one has
		\[
            \fp(n,\xc_p) = \xr^{-n} \int_{\fP} \Phi(\xr e(\xa)) e(-n\xa) \, d\xa + \xr_*^{-n} \int_{\fP_*} \Phi(\xr_*e(\xa)) e(-n\xa) \,d\xa + O(e^{(2-\xd)\rt{\fc n}}).
		\]
\end{proposition}

\begin{proof}
    Recalling from Lemma \ref{lem:prcPsi} that 
        \[
            \xY(\xr) = \fc X + 2\xl\logX + \xL + O(X^{-1/2}),
        \]
    with $\xl$ and $\xL$ given in said lemma (though the exact values are irrelevant here), Lemma \ref{lem:RePsiBB} implies the existence of some $\xd_1 < 1$ such that
        \<
        \label{eq:A-PInt}
        \begin{aligned}
            \xr^{-n} \int_{\fA \setminus \fP} \xF(\xr e(\xa))e(-n\xa) \,d\xa \less \xr^{-n} e^{\xd_1\xY(\rho)} \less \exp\lf(\tfr{n}{X}+\xd_1(\fc X+2\xl\logX)\rh).
        \end{aligned}
        \>
    We now note that the first equation in \eqref{eq:XX*ExactRecall} implies that $n/X = \fc X+2\xl$. As $\xd_1(\fc X+2\xl\logX) < \xd_2\fc X$ for some $\xd_1<\xd_2<1$ and all $X>X_0(\xd_1)$, again using \eqref{eq:XX*ExactRecall} we have
        \[
            \tf{n}{X} + \xd_1(\fc X + 2\xl\logX) < (1+\xd_2)\fc X + 2\xl = (1+\xd_2)\rt{\fc n+\xl^2} + (1-\xd_2)\xl.
        \]
    Since, as $n \to \infty$, one has
        \[
            (1+\xd_2)\rt{\fc n+\xl^2} + (1-\xd_2)\xl \sim (1+\xd_2)\rt{\fc n} < (2-\xd_3)\rt{\fc n}
        \]
    for some $\xd_3>0$, it follows that
        \[
            \exp(\tfr{n}{X}+\xd_1(\fc X+2\xl\logX)) \less \exp((2-\xd_3)\rt{\fc n}).
        \]
    From \eqref{eq:A-PInt} we conclude that
        \[
            \xr^{-n} \int_{\fA \setminus \fP} \exp(\xY(\xr e(\xa))-2\pi in\xa) \,d\xa \less \exp((2-\xd_3)\rt{\fc n}),
        \]
    and an analogous argument, using Lemmata \ref{lem:prcPsi*} and \ref{lem:RePsiBB} and equations \eqref{eq:XX*ExactRecall}, provides a similar bound for the integral over $\fA_* \setminus \fP_*$. By assuming, without loss of generality, that $\xd_3$ is small enough that $\fr32 < 2-\xd_3$, we obtain the result by taking $\xd = \xd_3$.
\end{proof}

\section{The principal integrals}
\label{sec:MainInts}

We now apply the formulae of sections \ref{sec:princ} and \ref{sec:relations} to derive asymptotic formulae for the integrals over $\fP$ and $\fP_*$. As the arguments for both integrals are so similar, and because we must later consider other ``principal'' integrals in sections \ref{sec:5} and \ref{sec:Kron}, we handle these cases simultaneously in the following general proposition.

\begin{proposition}
    \label{prop:MainInt}
    Let $0 \leq a \leq q$ and $(a,q)=1$, and let $\fc$, $\xk$, and $K$ be constants dependent on $p$, $q$ and $a$, with $\fc>0$. For $n>n_0$ let
        \<
            \label{eq:Xnlambda}
            X = \big(\rt{\fc n+\xk^2}-\xk\big)/\fc,
        \>
    and suppose that, for $n > n_0$ and $|\xb| \leq 3/(8\pi X)$, one has
        \<
            \label{eq:Psirhoalpha}
            \xY(\xr e(\tf{a}{q}+\xb)) = \fr{\fc X}{1-2\pi iX\xb} + 2\xk\log\pfr{X}{1-2\pi iX\xb} + K + O(X^{-\fr12}).
        \>
    Then, as $n\to\infty$, one has
        \<
            \label{eq:MainInt}
            \xr^{-n} \int_{\fP(a/q,X)} \Phi(\xr e(\xa))e(-n\xa) \,d\xa 
            = e\Big(\fr{-na}{q}\Big)\pfr{n}{\fc}^\xk e^{K} \fr{\fc^{\fr14}e^{2\rt{\fc n}}}{(4\pi)^{\fr12}n^{\fr34}} \lf[1+O(n^{-\fr15})\rh],
        \>
    where
        \[
            \fP(a/q,X) = \{ a/q + \xb : \xb \in \rr \text{ and } |\xb| \leq 3/(8\pi X) \}.
        \]
\end{proposition}

\begin{remark}
	In the following proof, quantities $c_0, c_1, \ldots$ denote positive constants.
\end{remark}

\begin{proof}
By \eqref{eq:Xnlambda}, as $n\to\infty$ we have 
    \<
        \label{eq:Xasymp}
        n/X = \fc X + 2\xk \qquad\text{and}\qquad X = (n/\fc)^{1/2} - \xk/\fc + O(n^{-1/2}).	
    \>
We start by writing integral \eqref{eq:MainInt} as
    \<
        \label{eq:prcInt}
        \int_{|\xb|\leq \fr{3}{8\pi X}} \xr^{-n} \exp\!\big[ \xY(\xr e(\tf{a}{q}+\xb)) - 2\pi in(\tf{a}{q}+\xb) \big] \,d\xb.
    \>
Writing $\xr e(\xb) = \exp(X^{-1}(1-2\pi iX\xb))$, we observe that under the change of variable $t = 2\pi X\xb$, formula \eqref{eq:Psirhoalpha} states that 
    \<
        \label{eq:Psiw}
        \xY\lf(e(\tf{a}{q})e^{\fr{1}{X}(1-it)}\rh) = \fr{\fc X}{1-it} + 2\xk \log\pfrac{X}{1-it} + K + O(X^{-\fr12}).
    \>
Making said change of variable in \eqref{eq:prcInt}, the integral therein becomes 
    \[
        \fr{e(-na/q)}{2\pi X} \int_{-\fr34}^{\fr34} \exp\lf[ \xY\lf(e^{\fr{1}{X}(1-it)}\rh) + (1-it)\fr{n}{X} \rh] \,dt.	
    \]
Now letting
    \[
        \xQ(t,X) := (1-it)\fr{n}{X} + \fr{\fc X}{1-it} + 2\xk\log\pfrac{X}{1-it},
    \]
formulae \eqref{eq:Xasymp} and \eqref{eq:Psiw} imply that integral \eqref{eq:prcInt} is
    \<
        \label{eq:prcInt3}
        = e\Big(\fr{-na}{q}\Big)\fr{e^K}{2\pi X} \lf[ 1+O(n^{-\fr14}) \rh] \int_{-\fr34}^{\fr34} \exp\xQ(t,X) \,dt.
    \>

In light of \eqref{eq:prcInt3}, we focus on the integral $\int_{-\fr34}^{\fr34} \exp\xQ(t,X) \,dt$. For $|t|<\tf34$ set
    \<
        \label{eq:xDvphi}
        \Delta := (1+t^2)^{-\fr12} \qquad\text{and}\qquad \varphi := \arctan t = t - \tf13t^3 + \tf15t^5 - \cdots,
    \>
so that
    \[
        (1-it)^{-1} = \Delta e^{i\varphi} \qquad\text{and}\qquad \log\pfrac{X}{1-it} = \logX\Delta + i\varphi.
    \]
With these, now
    \<
        \label{eq:xQForm}
        \xQ(t,X) = (1-it)(\fc X+2\xk) + \fc X\Delta e^{i\varphi} + 2\xk(\logX\Delta + i\varphi).
    \>
When $|t|<\tf34$ we have $\tf45 < \Delta \leq1$ and
    \[
        \Delta^2 = (1+t^2)^{-1} \leq 1-t^2+t^4 \leq 1-\tf25t^2,	
    \]
and because $\cos\xvf = \Delta$ and $X$ is large we find that
    \<
    \label{eq:xQReal}
    \begin{aligned}
        \Re\xQ(t,X) &\leq \fc X + 2\xk + \fc X(1-\tf25t^2) + 2\xk\logX + 2\xk\log{\xD} \\
        & \leq 2\fc X + 2\xk + 2\xk\logX -\tf25\fc Xt^2 + c_0.
    \end{aligned}	
    \>
If $|t|\geq X^{-\fr12}(\logX)^{\fr32}$ as well, then
    \<
        \label{eq:xQReal2}
        -\tf25\fc Xt^2 + c_0 < -\tf25\fc (\logX)^3 + c_0 < -c_1(\logX)^3.
    \>
Since $-c_1(\logX)^3 \sim -c_2(\log n)^3$ by \eqref{eq:Xasymp}, it follows from \eqref{eq:xQReal} and \eqref{eq:xQReal2} that
    \<
        \label{eq:prcIntErr}
    \begin{aligned}
    \int_{X^{-\fr12}(\log{X})^{\fr32} \leq |t| < \fr34} \exp\xQ(t,X) \,dt 
            &\less X^{2\xk}e^{2\fc X+2\xk}\exp(-c_1(\logX)^3)
            \less X^{2\xk}e^{2\fc X+2\xk}n^{-B} 
    \end{aligned}
    \>
for any $B > 0$ and $n > n_0(B)$.

In light of \eqref{eq:prcIntErr}, for the remainder of the proof it is convenient to let
    \[
        \xd := X^{-\fr12}(\logX)^{\fr32} \qquad\text{and}\qquad Y := 2(\fc X+\xk).
    \]
It thus remains to consider the integral in \eqref{eq:prcInt3} for $|t| < \xd$.
Returning to equation \eqref{eq:xQForm}, we now express $\Delta e^{i\varphi} = (1-it)^{-1}$ as a geometric series, and write $\varphi=t-\fr13t^3+\fr15t^5-\cdots$ as in \eqref{eq:xDvphi}, so that
    \begin{align*}
        \xQ(t,X) &= (1-it)(\fc X+2\xk) + \fc X(1+it-t^2+O(|t|^3)) \\ 
        &\qquad + 2\xk\pth{\logX\Delta + it + O(|t|^3)} \\
        &= Y + 2\xk\logX - \fc Xt^2 - \xk\log(1+t^2) + O(X\xd^3) \\
        &= Y + 2\xk\logX - \tf12Y t^2 + O(t^4) + O(X\xd^3).
    \end{align*}
As $t^4 \leq \xd^4 \leq X\xd^3$ and $X\xd^3 = X^{-\fr12}(\logX)^{\fr92} \less n^{-\fr15}$, we deduce that
    \[
        \xQ(t,X) = Y + 2\xk\logX - \tf12Y t^2 + O(n^{-\fr15}) \qquad (|t|<\xd),
    \]
and thus we have
    \<
    \label{eq:NewEq}
        \int_{-\xd}^{\xd} \exp\xQ(t,x) \,dt = X^{2\xk}e^{Y}\lf[1+O(n^{-\fr15})\rh] \int_{-\xd}^\xd e^{-\tf12Y t^2} \,dt.
    \>
As $Y \asymp X \asymp \rt{n}$ and $\int_\xd^\infty \exp(-\tf12Y t^2) \,dt \less \exp(-\tf12Y \xd^2)/(Y\xd)$, we have
    \[
        \int_{-\xd}^\xd e^{-\fr12Y t^2} \,dt - \rt{\fr{2\pi}{Y}}
        \less \fr{\exp(-c_3(\logX)^3)}{X^{1/2}(\logX)^{3/2}} 
        \less \exp(-c_4(\log n)^3) 
        \less n^{-B}
    \]
for $B>0$ and $n > n_0(B)$. Since $Y = 2\rt{\fc n+\xk^2} \sim 2\rt{\fc n}$ and $B$ can be arbitrarily large, it follows from equation \eqref{eq:NewEq} that
    \<
    \label{eq:prcMain}
    \begin{aligned}
        \int_{-\xd}^\xd \exp\xQ(t,X) \,dt &= X^{2\xk}e^Y\lf[1+O(n^{-\fr15})\rh] \bigg[ \rt{\fr{2\pi}{Y}} + O(n^{-B}) \bigg] \\
        &= X^{2\xk}e^{Y}\rt{\fr{2\pi}{Y}}\lf[1+O(n^{-\fr15})\rh]. 
    \end{aligned}
    \>

Using \eqref{eq:prcIntErr} and \eqref{eq:prcMain} in \eqref{eq:prcInt3} then, and again using the facts that $Y \sim 2\rt{\fc n}$ and that $B>0$ can be arbitrary, we deduce that
    \[
        e\Big(\fr{-na}{q}\Big) \fr{e^K}{2\pi X} \lf[ 1+O(n^{-\fr14}) \rh] \int_{-\fr34}^{\fr34} \exp\xQ(t,X) \,dt 
        = e\Big(\fr{-na}{q}\Big) e^{K} \fr{X^{2\xk-1}e^{Y}}{\rt{2\pi Y}} \lf[ 1 + O(n^{-\fr15}) \rh].
    \]
Finally, using the asymptotic relations
    \[
        X = \rt{n/\fc} - \xk/\fc + O(n^{-\nfr12}) \qquad\text{and}\qquad Y = 2\rt{\fc n} + O(n^{-1}),
    \]
we find that
    \[
        e\Big(\fr{-na}{q}\Big) e^{K}\fr{X^{2\xk-1}e^{Y}}{\rt{2\pi Y}} \lf[ 1 + O(n^{-\fr15}) \rh] 
        = e\Big(\fr{-na}{q}\Big) \pfr{n}{\fc}^{\xk}e^{K} \fr{\fc^{\fr14}e^{2\rt{\fc n}}}{(4\pi)^{\fr12}n^{\fr34}} \lf[ 1 + O(n^{-\fr15}) \rh],
    \]
completing the proof of the proposition.
\end{proof}

\section{A general asymptotic formula}
\label{sec:Asymptotics}

After briefly recalling the pertinent results and constants from the preceding sections, we now apply our results to derive a general asymptotic formula for $\fp(n,\xc_p)$, again excluding the cases of $p=2$ and $p=5$. Specifically, by separating primes by their residues (mod $4$) and by the values of $\xc_p(2)$, which is equivalent to considering the residues $p\mmod{8}$, we can simplify the constants in our asymptotic results.

We first recall the constants 
    \<
        \label{eq:kappaC}
        \xk = \pi\rt{\tf23} \qquad\text{and}\qquad \fc = \tfr{\pi^2}{24}(1-\tfr1p),
    \>
along with Hardy and Ramanujan's ``benchmark'' formula
    \[
        \fp(n,1) \sim (4\rt{3})^{-1}n^{-1}\exp\lf(\xk\rt{n}\rh)\lf[1+O(n^{-\fr15})\rh].
    \]
Again writing
    \[
        L_r(\xc) = L(r,\xc) \qquad\text{and}\qquad L_r'(\xc) = L'(r,\xc) \qquad (\text{for $r=0,1$}),
    \]
Lemmata \ref{lem:prcPsi} and \ref{lem:prcPsi*} establish that
    \begin{align}
        \label{eq:PsiPsi*Recall}
        \xY(\xr) &= \fc X + 2\xl\logX + \xL + O(X^{-\fr12}) \qquad (X\to\infty), \\
        \xY(-\xr_*) &= \fc X_* + 2\xl_*\log{X_*} + \xL_* + O(X_*^{-\fr12}) \qquad (X_*\to\infty),
    \end{align}
respectively, where 
    \begin{alignat}{3}
        \label{eq:xlrecall}
        \xl\hphantom{{}_*} &= \tf14L_0(\xc),\qquad &&\xL &&= \tf12\lf[ L_0'(\xc) + L_0(\xc)\log{2} \rh] - \tf14\log{p}, \\
        \label{eq:xl*recall}
        \xl_* &= \tf14(2\xc(2)-1)L_0(\xc), \qquad &&\xL_* &&= \tf12\lf[ (2\xc(2)-1) L_0'(\xc) - L_0(\xc)\log{2} \rh] - \tf14\log{p}.
    \end{alignat}
Again defining
    \<
        \label{eq:XX*Recall}
        X = (\rt{\fc n+\xl^2}-\xl)/\fc \qquad\text{and}\qquad X_* = (\rt{\fc n+\xl_*^2}-\xl_*)/\fc,
    \>
Proposition \ref{prop:pncReduc} establishes that
    \<
        \label{eq:pncReducedRecall}
        \fp(n,\xc) = \xr^{-n} \int_{\fP} \Phi(\xr e(\xa)) e(-n\xa) \, d\xa + \xr_*^{-n} \int_{\fP_*} \Phi(\xr_*e(\xa)) e(-n\xa) \,d\xa + O(e^{(2-\xd)\rt{\fc n}})
    \>
for some $\xd>0$. Using \eqref{eq:PsiPsi*Recall}--\eqref{eq:XX*Recall} then, Proposition \ref{prop:MainInt} provides asymptotic formulae for the integrals over $\fP$ and $\fP_*$ in \eqref{eq:pncReducedRecall}; in particular, Proposition \ref{prop:MainInt} implies that 
    \<
    \label{eq:pCrude}
    \begin{aligned}
        \fp(n,\xc) 
        &= \pfr{n}{\fc}^\xl e^{\xL} \fr{\fc^{\fr14}e^{2\rt{\fc n}}}{(4\pi)^{\fr12}n^{\fr34}} \lf[1+O(n^{-\fr15})\rh] + (-1)^n \pfr{n}{\fc}^{\xl_*} e^{\xL_*} \fr{\fc^{\fr14}e^{2\rt{\fc n}}}{(4\pi)^{\fr12}n^{\fr34}} \lf[1+O(n^{-\fr15})\rh] \\
            &\qquad + O(e^{(2-\xd)\rt{\fc n}}).
    \end{aligned}
    \>
Because the first term on the righthand side here is $\asymp n^{\xl-3/4}\exp(2\rt{\fc n})$, we absorb the error term $O(e^{(2-\xd)\rt{\fc n}})$ into this first term and collect constants to write
    \<
        \label{eq:GenAsymptoticCrude}
        \fp(n,\xc) = \fa_p n^{-\fr34} \exp(2\rt{\fc n}) \lf[n^\xl(1+O(n^{-\fr15})) + (-1)^n\fb_p n^{\xl_*}(1+O(n^{-\fr15})) \rh],
    \>
where
    \[
        \fa_p := (4\pi)^{-\fr12}\fc^{\fr14-\xl}e^{\xL} \qquad\text{and}\qquad \fb_p := \fc^{\xl-\xl_*}e^{\xL_*-\xL}.
    \]
For the moment, we assume that $L_0(\xc_{p}) \geq 0$ for all $p$. Then, using \eqref{eq:xlrecall} and \eqref{eq:xl*recall}, we have
    \<
        \label{eq:lambdaIneq}
        \xl_* - \xl = \tf12(\xc(2)-1)L_0(\xc) \leq 0,
    \>
so we further extract a factor of $n^\xl$ from the square bracketed term in \eqref{eq:GenAsymptoticCrude}, and rewrite
    \[
        2\rt{\fc n} = 2\rt{\tf{\pi^2}{24}(1-\tf1p)n} = \tf12\xk\rt{(1-\tf1p)n} \qquad \Big(\xk=\pi\rt{\tf23}\,\Big),
    \] 
to finally establish the asymptotic formula
    \<
        \label{eq:GenAsymptotic}
        \fp(n,\xc) = \fa_p n^{\xl-\fr34} \exp\lf(\tf12\xk\rt{(1-\tf1p)n}\,\rh) \lf[ 1 + (-1)^n \fb_p n^{\xl_*-\xl} + O(n^{-\fr15}) \rh],
    \>
where again
    \[
        \fa_p = (4\pi)^{-\fr12}\fc^{\fr14-\xl}e^{\xL} \qquad\text{and}\qquad \fb_p = \fc^{\xl-\xl_*}e^{\xL_*-\xL}.
    \]

In order to ``unpack'' the formulae for $\xl$, $\xL$, etc., we record several basic facts regarding Dirichlet $L$-functions. Using the well-known the functional equation  
    \[
        L(s,\xc) = \fr{G(\xc)}{i^\xn \pi} \lf(\fr{2\pi}{p}\rh)^{s} L(1-s,\bar{\xc}) \xG(1-s) \sin \tf{\pi}{2}(s+\xn), \quad \text{where} \quad \xn = \begin{cases} 0 & p \equiv  1 \mod4 \\ 1 & p \equiv  3 \mod 4 \end{cases},
    \]
see, e.g., \cite{montgomery2007multiplicative}*{Cor.\ 10.9}, one may quickly check that 
    \<
        \label{eq:L0}
        L_0(\xc) = \begin{cases}
            0 & p\equiv 1\mmod4, \\
            \tf{1}{\pi}\rt{p}L_1(\xc) & p\equiv 3\mmod4,
        \end{cases}	
    \>
and
    \<
        \label{eq:L0p1}
        L_0'(\xc) = \begin{cases}
            \tf12\rt{p}L_1(\xc) & p \equiv  1 \mmod{4}, \\
            \tfr{1}{\pi}\rt{p}L_1(\xc)\lf(\xg + \log(\fr{2\pi}{p}) - \fr{L_1'(\xc)}{L_1(\xc)}\rh) & p \equiv  3 \mmod{4},
        \end{cases}
    \>
where $\xg$ is the Euler-Mascheroni constant. Finally, one has
    \<
        \label{eq:L1pos}
        L_1(\xc_p) > 0
    \>
for all $p$, and indeed this inequality holds for all nonprincipal quadratic characters; see, e.g., \cite{montgomery2007multiplicative}*{p.\ 124}. We note that by \eqref{eq:L0} and \eqref{eq:L1pos}, our assumption that $L_0(\xc) \geq 0$ in \eqref{eq:lambdaIneq} is validated.

\section{Formulae for different residues modulo 8}
\label{sec:AsymptoticsMod8}

When $p\equiv 1\mmod{4}$, many of our computations and formulae are greatly simplified by the fact that $L_0(\xc) = 0$, since in this case $\xl = \xl_* = 0$, and thus formula \eqref{eq:GenAsymptotic} reduces to
    \<
        \label{eq:P14Crude}
        \fp(n,\xc) = \fa_p n^{-\fr34} \exp\!\Big(\tf12\xk\rt{(1-\tf1p)n}\,\Big) \lf[1 + (-1)^n\fb_p +O(n^{-\fr15}) \rh]
    \>
where
    \[
        \xk=\pi\rt{\tf23},\quad \fa_p = \ptfr{p-1}{384\,p}^{\fr14}e^{\xL}, \quad\text{and}\quad \fb_p = e^{\xL_*-\xL}.
    \]
From \eqref{eq:xlrecall} and \eqref{eq:L0p1} we have
    \[
        \xL = \tf12L_0'(\xc) - \tf14\log{p} = \tf14\rt{p}L_1(\xc) - \tf14\log{p},
    \]
so that, regardless of the value of $\xc(2)$, the constant $\fa_p$ has formula
    \<
        \label{eq:P14fa}
        \fa_p = \ptfr{p-1}{384\,p^2}^{\fr14}\exp(\tf14\rt{p}L_1(\xc)).
    \>
By \eqref{eq:xl*recall} and \eqref{eq:L0p1} we have 
    \[
        \xL_*-\xL = (\xc(2)-1)L_0'(\xc) = \tf12(\xc(2)-1)\rt{p}L_1(\xc),
    \]
and it follows that
    \<
        \label{eq:P14fb}
        \fb_p = e^{\xL_*-\xL} = \begin{cases}
            1 & p\equiv 1\mmod{8},\\
            \exp(-\rt{p}L_1(\xc)) & p\equiv 5\mmod{8}.
        \end{cases}
    \>
Thus, the formulae of Theorem \ref{thm:P14} are established.

\subsection{Formulae for \tops{$p \equiv  3\mmod{4}$}{p=3(mod 4)}}

When $p\equiv 3\mmod4$ our formulae become noticeably more complicated, especially with respect to the values of $\fa_p$ and $\fb_p$. Our asymptotic formulae when $p\equiv 3\mmod{4}$ and $\xc(2)=+1$, i.e., when $p\equiv 7\mmod{8}$, are somewhat simpler than those for $p\equiv 3\mmod{8}$, so we consider the former case first. 

When $p\equiv 7\mmod{8}$, from \eqref{eq:xlrecall}, \eqref{eq:xl*recall}, and \eqref{eq:L0} we have
    \[
        \xl = \xl_* = \tf14L_0(\xc) = \rt{p}L_1(\xc)/4\pi.
    \]
In this case then, we have
    \[
        \fp(n,\xc_p) = \fa_p n^{\rt{p}L_1(\xc)/4\pi-3/4} \exp\lf(\tf12\xk\rt{(1-\tf1p)n}\,\rh) \lf[ 1 + (-1)^n \fb_p + O(n^{-\nfr15}) \rh],
    \]
with
    \[
        \xk=\pi\rt{2/3}, \qquad \fb_p = e^{\xL_*-\xL} = e^{-L_0(\xc)\log{2}} = 2^{-\rt{p}L_1(\xc)/\pi},
    \]
and $\fa_p$ kept (for the moment) as
    \<
        \label{eq:fa34Raw}
        \fa_p = \ptfr{p-1}{384\,p}^{\fr14} \fc^{-\xl}e^{\xL}.
    \>
Now recalling \eqref{eq:L0} and \eqref{eq:L0p1}, we have
    \<
        \label{eq:eLambda34}
    \begin{aligned}
        \fc^{-\xl}e^{\xL} &= \exp\!\Big( \tfr12L_0'(\xc)+\tfr12L_0(\xc)\log{2}-\tfr14\log{p} -\tf12L_0(\xc)\log(\rt \fc) \Big) \\
        &= p^{-\fr14} \exp\lf( \fr{\rt{p}L_1(\xc)}{2\pi}\lf( \xg + \log\pfr{4\pi}{p\rt{\fc}} - \fr{L_1'(\xc)}{L_1(\xc)} \rh)\!\rh) \\
        &= p^{-\fr14} \exp\!\bigg( \fr{\rt{p}L_1(\xc)}{2\pi}\bigg( \xg + \fr12\log\!\bigg(\fr{384}{p-1}\bigg) - \fr{L_1'(\xc)}{L_1(\xc)} \bigg)\!\bigg).
    \end{aligned}
    \>
Combining \eqref{eq:fa34Raw} and \eqref{eq:eLambda34} then, we conclude that
    \<
        \label{eq:P78Form}
        \fp(n,\xc) = \fa_p n^{\rt{p}L_1(\xc)/4\pi-3/4} \exp\lf(\tf12\xk\rt{(1-\tf1p)n}\,\rh) \lf[ 1 + (-1)^n \fb_p + O(n^{-\fr15}) \rh]
    \>
with
    \<
        \label{eq:P78fa}
        \fa_p = \pfr{p-1}{384\,p^2}^{\fr14} \exp\!\bigg( \fr{\rt{p}L_1(\xc)}{2\pi}\bigg( \xg + \fr12\log\!\bigg(\fr{384}{p-1}\bigg) - \fr{L_1'(\xc)}{L_1(\xc)} \bigg)\!\bigg)
    \>
and
    \[
        \fb_p = 2^{-\rt{p}L_1(\xc)/\pi}.
    \]
This proves the first part of Theorem \ref{thm:P34}.

Finally we consider primes $p$ such that $p\equiv 3\mmod{4}$ and $\xc(2)=-1$, or equivalently primes such that $p\equiv 3\mmod{8}$. Now we find that
    \<
        \label{eq:38lambdadiff}
        \xl-\xl_* = \tf14L_0(\xc) + \tf34L_0(\xc) = \rt{p}L_1(\xc)/\pi
    \>
and
    \[
        \xL_*-\xL = -2L_0'(\xc) - L_0(\xc)\log{2} = \fr{\rt{p}L_1(\xc)}{\pi}\lf(2\fr{L_1'(\xc)}{L_1(\xc)} - 2\xg - 2\log\pfr{2\pi}{p} - \log{2} \rh).
    \]
While our formula for $\fa_p$ is the same as in \eqref{eq:P78fa}, we now have
    \<
    \begin{aligned}
        \fc^{\xl-\xl_*} = \exp\lf( \fr{\rt{p}L_1(\xc)}{\pi}\log\pfr{\pi^2(1-\nfr1{p})}{24}\!\rh), 
    \end{aligned}
    \>
and thus
    \<
    \label{eq:P38fb}
    \begin{aligned}
        \fb_p = \fc^{\xl-\xl_*}e^{\xL_*-\xL} 
        &= \exp\lf[\fr{\rt{p}L_1(\xc)}{\pi}\lf(2\fr{L_1'(\xc)}{L_1(\xc)} - 2\xg + \log\pfr{p(p-1)}{192} \!\rh)\rh].
    \end{aligned}
    \>
Additionally noting that now $n^{\xl_*-\xl} = n^{-\rt{p}L_1(\xc)/\pi}$ in \eqref{eq:GenAsymptotic}, we conclude that for $p \equiv  3 \mmod{8}$ we have
    \<
        \label{eq:P38Form}
		\fp(n,\xc) = \fa_p n^{\rt{p}L_1(\xc)/4\pi-\nfr34} \exp\lf(\tf12\xk\rt{(1-\tf1p)n}\,\rh) \lf[ 1+(-1)^n \fb_p n^{-\rt{p}L_1(\xc)/\pi} + O(n^{-1/5}) \rh],
	\>
where $\fa_p$ is again given by $\eqref{eq:P78fa}$, and $\fb_p$ is given by \eqref{eq:P38fb}, completing the proof of the second part of Theorem \ref{thm:P34}.

\begin{remark}
    Recalling the definitions
        \[
            X = (\rt{\fc n+\xl^2}-\xl)/\fc \qquad\text{and}\qquad X_* = (\rt{\fc n+\xl_*^2}-\xl_*)/\fc,
        \]
    and noting that $(\rt{\fc n+t^2}-t)/\fc$ is strictly decreasing in $t$, equation \eqref{eq:38lambdadiff} implies that $X_* > X$  (and thus $\xr_* > \xr$) when $p\equiv 3\mmod{8}$. Therefore, to compute $\fp(n,\xc_p)$ in this case we use a contour of integration like that in the following picture.
\end{remark}

{
\begin{figure}[!ht]
\centering
{\scalebox{0.95}{
\begin{tikzpicture}
    \def\gap{0.2}
    \def\tzR{2.35} 
    \def\tzr{1.80} 
    \def\tzudr{3} 
    
    \draw[line width = 0.5pt] (0,0) circle (\tzudr);
    \draw[dashed, line width = 0.5pt] (0,0) circle (\tzR);
    \draw[dashed, line width = 0.5pt] (0,0) circle (\tzr);
    
    \draw [help lines,-] (-1.2*\tzudr, 0)--(1.2*\tzudr, 0);
    \draw [help lines,-] (0, -1.2*\tzudr)--(0, 1.2*\tzudr);
    
    \draw[line width=1pt,
        decoration = {markings, mark = at position 0.6 with {\arrow[line width=1.2pt]{>}}},
        postaction = {decorate}] 
        (90:\tzR) arc (90:270:\tzR);
    
    \draw[line width=1pt, 
        decoration = {markings, mark = at position 0.6 with {\arrow[line width=1.2pt]{>}}}, postaction = {decorate}] 
        (-90:\tzR) -- (-90:\tzr);
    
    \draw[line width=1pt, 
        decoration = {markings, mark = at position 0.6 with {\arrow[line width=1.2pt]{>}}}, postaction = {decorate}] 
        (-90:\tzr) arc (-90:90:\tzr);
    
    \draw[line width=1pt, 
        decoration = {markings, mark = at position 0.6 with {\arrow[line width=1.2pt]{>}}}, postaction = {decorate}] 
        (90:\tzr) -- (90:\tzR);
    
    \draw[very thin, dashed] (0,0)--(125:\tzR);
    \node[below] at (130:0.6*\tzR) {$\rho_*$};
    
    \draw[very thin, dashed] (0,0)--(345:\tzr);
    \node[below] at (345:0.6*\tzr) {$\rho$};
    
    \node[below right] at (\tzudr,0) {$1$};
    \node[below left] at (-\tzudr,0) {$-1$};
    \node[above right] at (0, \tzudr) {$i$};
    \node[below right] at (0, -\tzudr) {$-i$};
\end{tikzpicture}
}}
\caption{The contour of integration used when $p \equiv 3\mmod{8}$. }
\label{fig:P38Contour}
\end{figure}
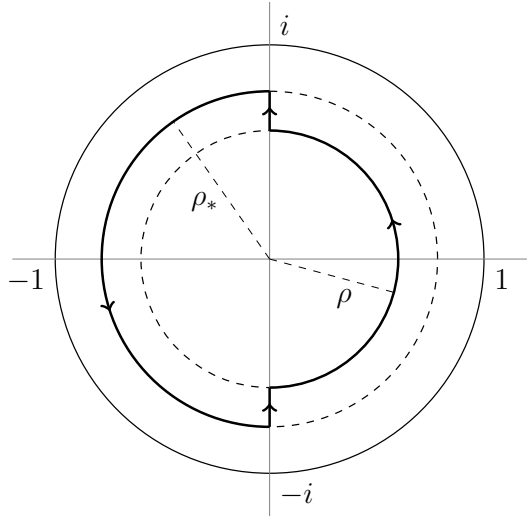
}

\section{The special case of \tops{$\fp(n,\xc_5)$}{p(n,chi5)}}
\label{sec:5}

The sequence $(\fp(n,\xc_5))_\nn$ requires special treatment because, roughly speaking, the generating function $\xF(z,\xc_5)$ has \emph{four} dominant singularities on the circle $|z|=1$ instead of the usual two at $z=\pm 1$. Specifically, $\xF(z,\xc_5)$ has dominant singularities $z = \pm 1$ and $z = e(\pm\tf15)$, as we demonstrate now.

It is immediate from Proposition \ref{prop:Major} that, as $X \to \infty$, for all odd $p$ one has
    \[
        \xY(\xr e(\tf{a}{q})) \sim V_q(a)\fc X,
    \]
where $\fc = \tf{\pi^2}{24}(1-\tf1p)$ and $V_q(a)$ is given in said proposition. We always have $V_1(1)=V_1(2)=1$, and, as seen in the proof of Lemma \ref{lem:RePsiBB}, for $p\neq 5$ with $p \equiv  1 \mmod{4}$ we have $V_q(a) < 1$ when $q > 2$ and $(a,q)=1$. This is why $\pm1$ (equivalently $e(0)$ and $e(\tf12)$) are usually the only dominant singularities of $\xF(z,\xc_p)$.

However, in the case of $p=5$ we have
    \[
        \fc = \fr{\pi^2}{30}, \quad G(\xc_5) = \rt{5}, \quad\text{and}\quad L_2(\xc_5) =\fr{4\pi^2\rt5}{125},
    \]
and thus for $q=5$ and $1\leq a\leq4$, the formula
    \[
        V_5(a) = \big[\xc_5(a)G(\xc_5)(1-\tf{\xc_5(2)}{4})L_2(\xc_5)/\fc - 1\big]\tfr{5}{25}
    \]
from \eqref{eq:Vqprime} reduces to
    \[
        V_5(a) = \tfr15(6\xc(a)-1) \qquad (1 \leq a \leq 4).
    \]
From this we see that
    \[
        V_5(1) = V_5(4) = 1 \qquad\text{and}\qquad V_5(2) = V_5(3) = -\tf75,
    \]
which implies that
    \[
        \xY(\xr) \sim \xY(-\xr) \sim \xY(\xr e(\tf15)) \sim \xY(\xr e(\tf45)) \sim \fc X,
    \]
illustrating the fact that $\xF(z,\xc_5)$ has four ``principal'' singularities.

\begin{remark}
    For brevity, going forward we identify $e(\tf45)=e(-\tf15)$ so that we may, e.g., speak of $V_5(\pm1)$, $\xY(\xr e(\pm\tf15))$, etc..
\end{remark}

Starting over with our ``crude'' formula \eqref{eq:pCrude} for $\fp(n,\xc_p)$, the additional principal arcs around $\xa = \pm\tf15$ (or formally $\xa=\tf15$ and $\xa=\tf45$) yield an asymptotic formula
    \<
    \begin{aligned}
        \label{5:eq:p5crude}
		\fp(n,\xc_5) = \fr{\fc^{\fr14}e^{2\rt{\fc n}}}{(4\pi)^{\fr12}n^{\fr34}}\lf[e^{\xL} + (-1)^n e^{\xL_*} + e^{-2\pi in/5}e^{i\xL_5(1)} + e^{2\pi in/5}e^{i\xL_5(-1)} + O(n^{-1/5}) \rh],
    \end{aligned}
    \>
where $\xL$ and $\xL_*$ have their usual values from \eqref{eq:xlrecall} and \eqref{eq:xl*recall}, and the quantities $\xL_5(\pm1)$ are constants determined below. As we ultimately find that \keep{$\xL_5(-1) = -\xL_5(1)$}, we collect constants and reduce \eqref{5:eq:p5crude} to
    \<
        \label{eq:p5crude2}
		\fp(n,\xc_5) = \fa_5 n^{-\fr34}\exp\!\Big(\tf12\xk\rt{\tf45n}\,\Big) \lf[ 1 + (-1)^n\fb_5 + \fd_5 \cos\lf(\tf{2\pi}{5} n-\xL_5(1)\rh) + O(n^{-\fr15}) \rh],
    \>
where $\xk=\pi\rt{2/3}$ as usual, where (from \eqref{eq:P14fa} and \eqref{eq:P14fb})
    \[
        \fa_5 = 480^{-\fr14}\exp(\tf14\rt{5}L_1(\xc_5)) \qquad\text{and}\qquad  \fb_5 = \exp(-\rt{5}L_1(\xc_5)),
    \]
respectively, and we now additionally define
    \[
        \fd_5 = 2\exp(-\xL) = 2\exp(-\tf14\rt{5}L_1(\xc_5)).
    \]
Computing (using, e.g., Lemma \ref{lem:MV10.1.14}) that
    \[
        \rt{5}L_1(\xc_5) = \log(\tf12(3+\rt{5})),
    \]
we find the explicit formulae
    \[
        \fa_5 = \bigg(\fr{3+\rt{5}}{960}\bigg)^{\fr14}, \quad \fb_5 = \fr{3-\rt{5}}{2}, \quad\text{and}\quad \fd_5 = \rt{2(5-\rt{5})}.
    \]
Although it remains to compute the constants $\xL_5(\pm1)$ in formula \eqref{eq:p5crude2}, said formula highlights how $\fp(n,\xc_5)$ is influenced by the 10-periodic sum in the square brackets there.

\subsection{The computation of \tops{$\xL_5(\pm1)$}{Lambda5(+-1)}}

Just as $\xL$ and $\xL_*$ were defined as the constant (i.e., order $X^0$) terms in our formulae for $\xY(\pm\xr e(\xb))$ (note $\xr=\xr_*$ here since $5\equiv 1\mmod{4}$) in Lemmata \ref{lem:prcPsi} and \ref{lem:prcPsi*}, respectively, the quantities $\xL_5(\pm1)$ are the constant terms in the asymptotic formulae for $\xY(\xr e(\pm\tf15+\xb))$. Our computations of $\xL_5(\pm1)$ are similar to the methods used in establishing Lemmata \ref{lem:prcPsi} and \ref{lem:prcPsi*}, so we are brief in our exposition. Again we use the dashed integral $\dint$ as a shorthand for $\fr{1}{2\pi i} \int$, and we now reserve $\xc_0$ for the principal character (mod 5) and reserve $\xy$ for a general character (mod 5).

Let $|\xb| \leq 3/(8\pi X)$. Recalling the definitions of $\xY_0$ and $\xY_1$ from \eqref{eq:Psi0Psi1Def}, we again write $\xt = X^{-1}(1-2\pi iX\xb)$ so that $\xr e(\xb) = e^{-\xt}$, and we first consider
    \[
        \xY_0(e(\tfr{1}{5})e^{-\xt}) 
		= \sum_{2|k} \sum_{n=1}^\infty \frac{\xc_0(n)}{k} \e{\fr{nk}{5}} e^{-nk\xt} 
		= \sum_{2|k}\sum_{m=1}^5 \sum_{n \equiv  m\mod{5}} \frac{\xc_0(n)}{k} \e{\fr{mk}{5}} e^{-nk\xt}.   
    \]
Since the only nonzero sums $\sum_{n \equiv  m \mod5}$ here are those with $(m,5)=1$, we have 
	\<
		\label{eq:PsilSum}
		\sum_{m=1}^5 \sum_{n \equiv  m\mod{5}} \xc_0(n) e\pfr{mk}{5} e^{-nk\xt}
		= \fr{1}{4} \sum_{\xy\,\mathrm{mod}\,5} \sum_{n=1}^\infty \xy(n)\xc_0(n)e^{-nk\xt} \sum_{m=1}^5 \bar{\xy}(m) e\pfr{mk}{5},
	\>
and since $\xy(n)\xc_0(n)=\xy(n)$ for any character $\xy \mmod{5}$, it follows that 
	\[
		\sum_{2|k} \sum_{n=1}^\infty \frac{\xc_0(n)}{k} \e{\fr{nk}{5}} e^{-nk\xt} = \fr{1}{4} \sum_{\xy \,\mathrm{mod}\,5}	\sum_{2|k} \sum_{n=1}^\infty \fr{G(k,\bar{\xy})}{k} \xy(n)e^{-nk\xt},
	\]
where $G(k,\bar{\xy}) = \sum_{m=1}^5 \bar{\xy}(m) e(mk/5)$.

Writing $e^{-nk\xt} = \ldint{2} \xG(s)(nk\xt)^{-s} \,ds$ and rearranging terms, it follows that
    \<
    \begin{aligned}
        \label{eq:5:Psi0ints}
        &\xY_0(e(\tf{1}{5})e^{-\xt})
        = \fr{1}{4} \sum_{\xy\,\mathrm{mod}\,5} \ldint{2} \Big( \sum_{2|k} \fr{G(k,\bar{\xy})}{k^{s+1}} \Big) L(s,\xy)\xG(s)\xt^{-s} \,ds \\
        &\qquad = \fr14 \sum_{\xy\,\mathrm{mod}\,5} \sum_{\xn=1}^5 G(2\xn,\bar{\xy}) \ldint{2} 10^{-s} \xz(s+1,\tf{\xn}{5}) L(s,\xy) \xG(s)\xt^{-s} \,ds,
    \end{aligned}
    \>
where $\xz(z,r)$ is the Hurwitz zeta function. Akin to the arguments of section \ref{sec:princ}, using \eqref{eq:5:Psi0ints} and shifting the line of integration to $\Re(s)=-\fr12$, and bounding the resulting integral yields the formula
    \<
        \label{eq:5:Psi0Ints2}
        \xY_0(e(\tfr{1}{5})e^{-\xt}) = - \tfr{1}{5}\fc \xt^{-1} + \tfr{1}{4}\sum_{\substack{\xy\,\mathrm{mod}\,5 \\ \xy \neq \xc_0}} G(\bar{\xy})\tf{\xy(2)}{2}L_0(\xy)L_1(\xy) +  O(|\xt|^{\fr12}) \qquad (\fc=\tfr{\pi^2}{30}).
    \>
Using, e.g., Lemma \ref{lem:MV10.1.14} one may quickly reduce \eqref{eq:5:Psi0Ints2} to
    \begin{align}
        \label{5:eq:Psi0Form}
        \xY_0(e(\tfr{1}{5})e^{-\xt}) &= - \tf15\fc \xt^{-1} + 0 +  O(|\xt|^{\fr12}).
    \end{align}

Now turning to $\xY_1(\xr e(\fr15+\xb))$, in place of \eqref{eq:PsilSum} we have
	\[
		\sum_{2 \ssnmid k} \sum_{n=1}^\infty \frac{\xc_5(n)}{k} \e{\fr{nk}{5}} e^{-nk\xt} = \fr{1}{4} \sum_{\xy\,\mathrm{mod}\,5} \sum_{2 \ssnmid k} \sum_{n=1}^\infty \fr{G(k,\bar{\xy})}{k} \xc_5(n)\xy(n)e^{-nk\xt},	
	\]
and in place of \eqref{eq:5:Psi0ints} we have 
	\<
        \label{eq:5:Psi1Int}
		\xY_1(e(\tf{1}{5})e^{-\xt})
		= \fr{1}{4} \sum_{\xy\,\mathrm{mod}\,5} \ldint{2} \Big( \sum_{2 \ssnmid k} \fr{G(k,\bar{\xy})}{k^{s+1}} \Big) L(s,\xc_5\xy)\xG(s)\xt^{-s} \,ds.
	\>
Rather than write \eqref{eq:5:Psi1Int} using $\xz(z,r)$ like in \eqref{eq:5:Psi0ints}, it is perhaps simpler, and is indeed more illustrative, to compute that 
    \[
        \sum_{2 \ssnmid k} \fr{G(k,\bar{\xy})}{k^{s+1}} = \begin{cases}
            (1-2^{-s-1})(5^{-s}-1)\xz(s+1) & \xy = \xc_0, \\
            G(\bar{\xy})(1-\tf{\xy(2)}{2^{s+1}})L(s+1,\xy) & \xy \neq \xc_0
        \end{cases}
    \]
and use this to more closely examine the integrals in \eqref{eq:5:Psi1Int}.

This time, because of the factor $L(s,\xc_5\xy)$ in \eqref{eq:5:Psi1Int}, our main term there comes from the term with $\xy=\xc_5$, since $L(s,\xc_5^2) = L(s,\xc_0) = (1-5^{-s})\xz(s)$. Thus, writing the summands in \eqref{eq:5:Psi1Int} over $\xy \mmod{5}$ with $\xy$ in the order $\xc_5$, then $\xc_0$, and then the remaining characters (mod $5$), we have
	\<
    \label{eq:5:Psi1Int2}
	\begin{aligned}
		&\xY_1(e(\tfr{1}{5})e^{-\xt}) 
        = \fr{1}{4} \ldint{2} G(\xc_5)\big(1-\tf{\xc_5(2)}{2^{s+1}}\big)(1-5^{-s})L(s+1,\xc_5)\xz(s)\xG(s)\xt^{-s} \,ds \\
		& \qquad\qquad - \fr14 \ldint{2} (1-2^{-s-1})(1-5^{-s})\xz(s+1)L(s,\xc_5)\xG(s)\xt^{-s} \,ds \\
		& \qquad\qquad + \fr{1}{4}\sum_{\substack{\xy\,\mathrm{mod}\,5 \\ \xy \neq \xc_0,\xc_5}} G(\bar{\xy}) \ldint{2} \big(1-\tfr{\xy(2)}{2^{s+1}}\big)L(s+1,\xy)L(s,\xc_5\xy)\xG(s)\xt^{-s} \,ds.    
	\end{aligned}    
	\>

At $s=1$ the first integrand here has residue 
    \<
        \label{eq:5:Psi1Res1}
        \tf15G(\xc_5)(1-\tf{\xc_5(2)}{4})L_2(\xc_5)\xt^{-1} = \tfr{\pi^2}{25}\xt^{-1} = \tf65\fc \xt^{-1},
    \>
and at $s=0$ it has residue 0 due to the factor $(1-5^{-s})\xG(s)$. We remark that the leftmost expression in \eqref{eq:5:Psi1Res1} is consistent with formula \eqref{eq:Vqprime}. The second integrand in \eqref{eq:5:Psi1Int2} has residue $-\tf18L_0(\xc_5)\log{5}$ at $s=0$, but this is 0 since $L_0(\xc_5)=0$.

For the remaining integrals in \eqref{eq:5:Psi1Int2}, there are only two characters $\xy\mmod{5}$ different from $\xc_0$ and $\xc_5$, say $\xo$ and $\bar{\xo}$, where $\xo(2)=i$. It is again straightforward to compute that these two integrals together provide a residue at $s=0$ of
    \[
        \tf14G(\bar{\xo})(1-\tf{i}{2})L_0(\bar{\xo})L_1(\xo) + \tf14G(\xo)(1+\tf{i}{2})L_0(\xo)L_1(\bar{\xo}) = \tfr{i\pi}{10},
    \]
so that, in total, from \eqref{eq:5:Psi1Int2} we find that
    \<
        \label{5:eq:Psi1Form}
		\xY_1(e(\tfr{1}{5})e^{-\xt}) = \tfr65\fc \xt^{-1} + \tfr{i\pi}{10} + O(|\xt|^{\fr12}).
	\>

Adding \eqref{5:eq:Psi0Form} and \eqref{5:eq:Psi1Form} and replacing $\xt^{-1}$ with $X/(1-2\pi iX\xb)$, we have
	\[
		\xY(\xr e(\tf15+\xb)) = \fr{\fc X}{1-2\pi iX\xb} + \fr{i\pi}{10} + O(X^{-\fr12}),
	\]
and we thus deduce that
    \[
        \xL_5(1) = \tfr{\pi}{10}.
    \]
A similar computation provides essentially the same formula for $\xY(\xr e(-\tf15+\xb))$, with the only difference being the replacement of $i\pi/10$ with $-i\pi/10$, so that indeed
    \[
        \xL_5(-1) = -\xL_5(1) = -\tfr{\pi}{10}.
    \]

Having computed that $\xL_5(\pm1) = \pm\fr{\pi}{10}$, we return to \eqref{eq:p5crude2} and conclude that
    \<
        \label{eq:p5Form}
		\fp(n,\xc_5) = \fa_5 n^{-\fr34}\exp\!\Big(\tf12\xk\rt{\tf45n}\,\Big) \lf[ 1 + (-1)^n\fb_5 + \fd_5 \cos\lf(\tf{2\pi}{5}n-\tf{\pi}{10}\rh) + O(n^{-\fr15}) \rh],
    \>
where we recall that
    \[
        \xk = \pi\rt{\tf23}, \quad \fa_5 = \bigg(\fr{3+\rt{5}}{960}\bigg)^{\fr14}, \quad \fb_5 = \fr{3-\rt{5}}{2}, \quad\text{and}\quad \fd_5 = \rt{2(5-\rt{5})}.
    \]
This completes the proof of Theorem \ref{thm:P5}.

\section{The formula for \tops{$\fp(n,\xc_2)$}{p(n,chi2)}}
\label{sec:Kron}

We now consider this paper's final case, recalling that $\xc_2(n)$ is the Kronecker symbol $(\fr{n}{2})$ defined via
    \<
        \label{eq:2:Kronecker}
        \xc_2(n) = \begin{cases}
            1 & n \equiv  \pm1\mmod{8}, \\
            -1 & n\equiv  \pm3\mmod{8}, \\
            0 & \text{otherwise}.
        \end{cases}
    \>
Because $\xc_2$ is multiplicative and clearly $|\xc_2|\leq1$, Theorem \ref{thm:Minor} precludes the need for any minor arc analysis of $\xY(\xr e(\xa),\xc_2)$. Arguing as done in section \ref{sec:major}, for $\xa = a/q + \xb$ with $(a,q)=1$ and $|\xb| \leq 1/(qX^{2/3})$, we have
    \[
        \xY(\xr e(\xa),\xc_2) = \lf(\fr{X}{1-2\pi iX\xb}\rh)\sum_{k=1}^\infty \fr{U_k(q,a)}{k^2} + O(X^{\fr89}),
    \]
where now, because $\xc_2$ is a character modulo $8$, we define
    \<
        \label{eq:2:Uk}
        U_k(q,a) = \fr{1}{[8,q]}\sum_{m=1}^{[8,q]} \xc_2^k(m) e\lf(\fr{mka}{q}\rh) \qquad ([8,q]=\mathrm{lcm}(8,q)).
    \>

Since our prime of interest is $2$ here, we define (cf.\ equation \eqref{eq:cDef})
    \[
        \fc = \fr{\pi^2}{24}\lf(1-\fr{1}{2}\rh) = \fr{\pi^2}{48} \qquad \text{and}\qquad V_q(a) = \fr{1}{\fc} \sum_{k=1}^\infty \fr{U_k(q,a)}{k^2},
    \]
using formula \eqref{eq:2:Uk} for $U_k(q,a)$. Following computations akin to those done for Proposition \ref{prop:Major} (see the appendix), we derive the following analogue of said proposition. We recall that for any prime $\l$ and any $m>0$, one writes $\l^m \|n$ if one has $\l^m \mid n$ but $\l^{m+1} \nmid n$.

\begin{lemma}
    \label{lem:2:Vq}
    Let $\xc_2$ be defined as in \eqref{eq:2:Kronecker}, let $\fc = \fr{\pi^2}{48}$, and let $X > X_0$. For coprime $a$ and $q$ with $0 \leq a \leq q \leq X^{1/3}$, there exist constants $V_q(a)$ such that: for all $\xa = a/q + \xb$ with $|\xb| \leq 1/(qX^{2/3})$, as $X \to \infty$ one has
    \[
        \xY(\xr e(\xa),\xc_2) = \fr{V_q(a)\fc X}{1-2\pi iX\xb} + O(X^{\fr89}),
    \]
    where $V_q(a)$ is determined as follows. If $2^3\|q$ then {\normalfont (cf.\ \eqref{eq:Vqprime})}
        \<
            \label{eq:2:Vq8}
            V_q(a) = \big[ \xc_2(a)G(\xc_2)L_2(\xc_2)/\fc - 1\big]\fr{8}{q^2},
        \>
    and otherwise
        \<
            \label{eq:8:VqCases}
            V_q(a) = \begin{cases}
                1/q^2 & 2\nmid q, \\
                4/q^2 & 2^1\|q, \\
                -8/q^2 & 2^2\|q, \\
                -8/q^2 & 2^4|q.
            \end{cases}
        \>
\end{lemma}

\begin{remark}
    Because $(1-\tf{\xc_2(2)}{4})=1$, if we momentarily relax the restriction that $p$ only denote an odd prime and let $p=2$, equation \eqref{eq:2:Vq8} may be written as
    \[
        V_q(a) = \lf[ \xc_2(a)G(\xc_2)(1-\tf{\xc_2(2)}{4})L_2(\xc_2)/\fc - 1\rh]\fr{4p}{q^2},
    \]
and the quantities $-8/q^2$ in \eqref{eq:8:VqCases} may be written as $-4p/q^2$, further illustrating the similarities between Lemma \ref{lem:2:Vq} and Proposition \ref{prop:Major}. 
\end{remark}

As quick computations show that
    \[
        G(\xc_2) = 2\rt{2} \qquad\text{and}\qquad L_2(\xc_2) = \fr{\pi^2\rt{2}}{16},
    \]
we find that
    \[
        V_1(1)=V_2(1) = 1 \qquad\text{and}\qquad V_8(1) = V_8(7) =  \tf{11}{8},
    \]
showing that the major arcs about $\xa=\tf18$ and $\xa=\tf78$ in fact yield higher-order asymptotic terms than do the arcs about $0$, $1$, and $\tf12$. Maintaining that $\fc=\tf{\pi^2}{48}$, computations like those done in sections \ref{sec:princ} and \ref{sec:5} show that
    \<
        \label{eq:2:Psi18}
        \xY(\xr e(\pm\tf18+\xb),\xc_2) = \fr{\fr{11}{8}\fc X}{1-2\pi iX\xb} \pm \fr{3\pi i}{16} + O(X^{-\fr12})
    \>
and
    \<
        \label{eq:2:Psipm1}
        \xY(\pm\xr e(\xb),\xc_2) = \fr{\fc X}{1-2\pi iX\xb} + \lf(-\tf14\log{2} \pm \tf12\log(1+\rt{2})\rh) + O(X^{-\fr12}).
    \>

In section \ref{sec:relations}, the formulae of the form
    \<
        \label{eq:2:PsiPrcGen}
        \xY(\xr e(\xb),\xc_p) = \fr{\fc X}{1-2\pi iX\xb} + 2\xl\log\pfr{X}{1-2\pi iX\xb} + \xL + O(X^{-\fr12})
    \>
led us to define $X$ (and therefore the radius $\xr=e^{-1/X}$) via
    \[
        X = (\rt{\fc n+\xl^2}-\xl)/\fc,
    \]
and similarly for $X_*$. Comparing \eqref{eq:2:Psipm1} and \eqref{eq:2:PsiPrcGen}, we see that in this case $\xl=\xl_*=0$, so that for $\fp(n,\xc_2)$ we set
    \<
        \label{eq:2:XX*}
        X = X_* = \rt{n/\fc}.
    \>
On the other hand, comparing \eqref{eq:2:Psi18} and \eqref{eq:2:PsiPrcGen} we deduce that we should define
    \<
        \label{eq:2:X8}
        X_8 = \rt{8n/(11\fc)} \qquad\text{and}\qquad \xr_8 = e^{-1/X_8},
    \>
and use $\xr$ and $\xr_8 $ for an ``arc-transference contour'' for
    \[
        \fp(n,\xc_2) = \fr{1}{2\pi i} \int_{|z|=r} \xF(z,\xc_2) z^{-n-1} \,dz \qquad (0<r<1).
    \]
This time, our modified contour has the form shown in Figure \ref{fig:KronContour}, cf.\ Figure \ref{fig:P38Contour}. We remark that the inequality $\xr_8 < \xr$ implied by \eqref{eq:2:XX*} and \eqref{eq:2:X8} is reflected in the figure.

{
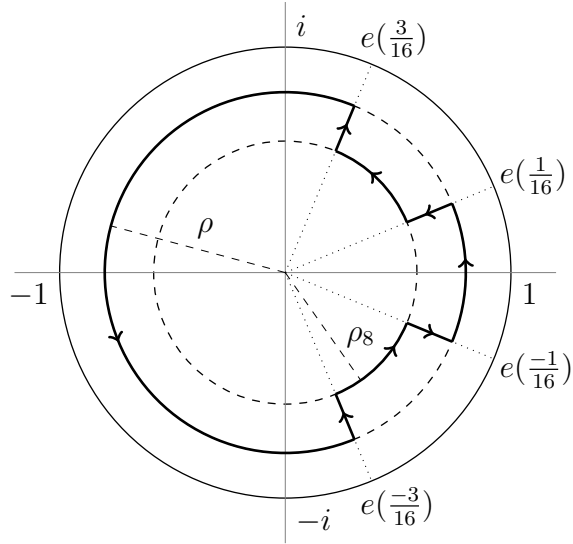
\begin{figure}[!ht]
\centering
\begin{tikzpicture}
    \def\gap{0.2}
    \def\tzR{2.4} 
    \def\tzr{1.75} 
    \def\tzudr{3} 
    
    \draw[line width = 0.5pt] (0,0) circle (\tzudr);
    \draw[dashed, line width = 0.5pt] (0,0) circle (\tzr);
    \draw[dashed, line width = 0.5pt] (0,0) circle (\tzR);
    
    \draw [help lines,-] (-1.2*\tzudr, 0)--(1.2*\tzudr, 0);
    \draw [help lines,-] (0, -1.2*\tzudr)--(0, 1.2*\tzudr);
    
    
    \draw[line width=1pt,
        decoration = {markings, mark = at position 0.6 with {\arrow[line width=1.2pt]{>}}},
        postaction = {decorate}] 
        (-22.5:\tzR) arc (-22.5:22.5:\tzR);

    \draw[line width=1pt, 
        decoration = {markings, mark = at position 0.6 with {\arrow[line width=1.2pt]{>}}}, postaction = {decorate}] 
        (22.5:\tzR) -- (22.5:\tzr);

    \draw[line width=1pt,
        decoration = {markings, mark = at position 0.6 with {\arrow[line width=1.2pt]{>}}},
        postaction = {decorate}] 
        (22.5:\tzr) arc (22.5:67.5:\tzr);

    \draw[line width=1pt, 
        decoration = {markings, mark = at position 0.6 with {\arrow[line width=1.2pt]{>}}}, postaction = {decorate}] 
        (67.5:\tzr) -- (67.5:\tzR);

    \draw[line width=1pt,
        decoration = {markings, mark = at position 0.6 with {\arrow[line width=1.2pt]{>}}},
        postaction = {decorate}] 
        (67.5:\tzR) arc (67.5:292.5:\tzR);
    
    \draw[line width=1pt, 
        decoration = {markings, mark = at position 0.6 with {\arrow[line width=1.2pt]{>}}}, postaction = {decorate}] 
        (-67.5:\tzR) -- (-67.5:\tzr);
    
    \draw[line width=1pt,
        decoration = {markings, mark = at position 0.75 with {\arrow[line width=1.2pt]{>}}},
        postaction = {decorate}] 
        (-67.5:\tzr) arc (-67.5:-22.5:\tzr);
    
    \draw[line width=1pt, 
        decoration = {markings, mark = at position 0.6 with {\arrow[line width=1.2pt]{>}}}, postaction = {decorate}] 
        (-22.5:\tzr) -- (-22.5:\tzR);
    
    \draw[very thin, dashed] (0,0)--(165:\tzR);
    \node[right] at (156:0.6*\tzR) {$\rho$};
    
    \draw[very thin, dashed] (0,0)--(305:\tzr);
    \node[right] at (308:0.62*\tzr) {$\rho_8$};
    
    \draw[dotted] (0,0)--(22.5:\tzudr);
    \node[right] at (25.5:\tzudr) {$e(\tf{1}{16})$};

    \draw[dotted] (0,0)--(67.5:\tzudr);
    \node[above right] at (72.5:0.95*\tzudr) {$e(\tf{3}{16})$};

    \draw[dotted] (0,0)--(-22.5:\tzudr);
    \node[right] at (-25.5:\tzudr) {$e(\tf{-1}{16})$};

    \draw[dotted] (0,0)--(-67.5:\tzudr);
    \node[below right] at (-72.5:0.95*\tzudr) {$e(\tf{-3}{16})$};
    
    \node[below right] at (\tzudr,0) {$1$};
    \node[below left] at (-\tzudr,0) {$-1$};
    \node[above right] at (0, \tzudr) {$i$};
    \node[below right] at (0, -\tzudr) {$-i$};

\end{tikzpicture}
\caption{The contour of integration used for computing $\fp(n,\chi_2)$.}
\label{fig:KronContour}
\end{figure}
}

The remaining derivations and estimates used to establish the formulae of section \ref{sec:MainInts} and \ref{sec:Asymptotics} may again be applied to derive similar results for $\fp(n,\xc_2)$. Because our principal arcs are now those about $\xa=\tf18$ and $\xa=\tf78$ (alternatively about $\xa=\pm\tf18$), our main asymptotic term for $\fp(n,\xc_2)$ is
    \<
    \label{eq:2:P18Term}
    \begin{aligned}
        & \ptfr{11}{8}^{\fr14}(4\pi)^{-\fr12}\fc^{\fr14}n^{-\fr34} \exp\!\Big(2\rt{\tfr{11}{8}\fc n}\,\Big) \lf[e\Big(\fr{-n}{8}\Big)e^{3\pi i/16} + e\pfr{n}{8}e^{-3\pi i/16}\rh] \\
        &\qquad = \fa_2 n^{-\fr34} \exp\!\Big(\tf12\xk\rt{\tf{11}{16}n}\,\Big) \cos\!\Big(\fr{2\pi n}{8}-\fr{3\pi}{16}\Big),
    \end{aligned}
    \>
where again $\xk=\pi\rt{2/3}$, and now
    \[
        \fa_2 = 2(\tf{11}{8})^{\fr14}(4\pi)^{-\fr12}\fc^{\fr14} = \ptfr{11}{384}^{\fr14}.
    \]

In light of \eqref{eq:2:Psipm1}, the principal arcs $\fP$ and $\fP_*$ contribute terms
    \<
        \label{eq:2:Ppm1Term}
        (4\pi)^{-\fr12}\fc^{\fr14}e^{\xL_1}n^{-\fr34}\exp\!\Big(\tfr12\xk\rt{\tfr12n}\,\Big) \quad\text{and}\quad
        (-1)^n(4\pi)^{-\fr12}\fc^{\fr14}e^{\xL_2}n^{-\fr34}\exp\!\Big(\tfr12\xk\rt{\tfr12n}\,\Big),
    \>
respectively, where, breaking with our previous notation just this once, we simply let
    \[
        \xL_1 = -\tf14\log{2} + \tf12\log(1+\rt2) \qquad\text{and}\qquad \xL_2 = -\tf14\log{2} - \tf12\log(1+\rt2).
    \]
As $\exp(\xL_2-\xL_1) = (1+\rt{2})^{-1}$, we sum up the quantities \eqref{eq:2:P18Term} and \eqref{eq:2:Ppm1Term} and collect constants to conclude that
    \[
        \fp(n,\xc_2) = \fa_2 n^{-\fr34} \exp\!\Big(\tf12\xk\rt{\tf{11}{16}n}\,\Big) \fS(n),
    \]
where
    \[
        \fS(n) = \cos\!\Big(\fr{2\pi n}{8}-\fr{3\pi}{16}\Big)
        + \fb_2 \lf[1 + \fr{(-1)^n}{1+\rt2}\rh] \exp\!\Big(\!\!-\!\tfr12\xk\rt{\tf12n}\Big(\rt{\tfr{11}{8}}-1\Big)\!\Big),
    \]
    \[
        \xk=\pi\rt{\tf23}, \quad \fa_2 = \lf(\fr{11}{384}\rh)^{1/4}, \quad\text{and}\quad \fb_2 = \rt{\fr{1+\rt2}{2\rt{11}}}.
    \]
This completes the proof of Theorem \ref{thm:Kron}.

\appendix

\section{The computations for Proposition \ref{prop:Major}}
\label{sec:App1}
Here in the first part of the appendix we include the computations used in establishing Proposition \ref{prop:Major}; paraphrasing, we recall Proposition \ref{prop:Major} states that 
    \[
        \xY(\xr e(\tf{a}{q}+\xb),\xc_p) = \fr{V_q(a)\fc X}{1-2\pi iX\xb} + O(X^{\fr89}) \qquad (\text{for $(a,q)=1$}),
    \]
where 
    \[
        \fc = \fr{\pi^2}{24}\lf(1-\fr{1}{p}\rh), \qquad V_q(a) = \fr{1}{\fc} \sum_{k=1}^\infty \fr{U_k(q,a)}{k^2},
    \]
and
    \<
        \label{eq:UkRecall}
        U_k(q,a) = \fr{1}{[p,q]}\sum_{m=1}^{[p,q]} \xc^k(m) e\lf(\fr{mka}{q}\rh) \qquad ([p,q]=\mathrm{lcm}(p,q)).
    \>
As mentioned in section \ref{sec:major}, the specific value of $V_q(a)$ depends on divisibility properties of $q$, and thus we consider four different cases regarding $q$.

\begin{lemma}
    If $(p,q) = (a,q) = 1$, then
        \[
            V_q(a) = \begin{cases}
                1/q^2 & 2\nmid q,\\
                4/q^2 & 2\mid q.
            \end{cases}
        \]
\end{lemma}

\begin{proof}
    First, by assumption we have $[p,q]=pq$. Writing $m=tp+s$ with $1 \leq t \leq q$ and $1 \leq s \leq p$ in \eqref{eq:UkRecall}, we find at once that
        \<
            \label{eq:NewEq2}
			pq U_k(q,a) = \sum_{s=1}^p \sum_{t=1}^q e\pf{(tp+s)ka}{q}\xc^k(tp+s) = \sum_{s=1}^p \xc^k(s) e\pf{ska}{q} \sum_{t=1}^q e\pfr{tkpa}{q},
		\>
    and it follows at once that
        \<
            \label{eq:NewEq3}
            U_k(q,a) = \begin{cases} (1-\tf1p) & \text{$2\mid k$ and $q\mid k$,} \\
            0 & \text{otherwise.}
            \end{cases}
        \>
    Now since $U_k(q,a) = 0$ unless $q\mid k$, we write $k=qk_1$ for such $k$ to find that
		\[
			\sum_{k=1}^\infty \fr{U_k(q,a)}{k^2} = \left(1-\fr{1}{p}\right) \sum_{\substack{k_1=1 \\ 2 \ssmid qk_1}}^\infty \fr{1}{q^2k_1^2},
		\]
    and the result follows immediately from simply conditioning on the parity of $q$. 
\end{proof}

We now suppose that $p\mid q$, recalling that $G(\xy)=\sum_{n=1}^q \bar{\xy}(n) e(n/q)$ for a general character $\xy$ (mod $q$).

\begin{lemma}
	\label{lem:UkDiv}
	Suppose that $(a,q)=1$ and that $p\mid q$, say $q=pq_1$. First, one has
        \[
            U_k(q,a) = 0 \qquad \text{if $q_1\nmid k$.}
        \]
    Now supposing that $k=q_1k_1$ for some $k_1\in\nn$, one has
		\<
			\label{eq:UkDiv}
			U_k(q,a) = \begin{cases} 
            1-1/p & \text{$2\mid k$ and $p\mid k_1$,} \\
            - 1/p & \text{$2\mid k$ and $p\nmid k_1$,} \\
            G(\xc)\xc(ak_1)/p & \text{$2\nmid k$.}
            \end{cases}
		\>
\end{lemma}

\begin{proof}
    Now with $[p,q]=q$, similar to equation \eqref{eq:NewEq2} we find at once that
		\[
			q U_k(q,a) = \sum_{t=1}^{q_1} \sum_{s=1}^p \xc^k(tp+s) e\pfr{(tp+s)ka}{q} 
			= \sum_{s=1}^{p} \xc^k(s) \pfr{ska}{q} \sum_{t=1}^{q_1} e\pfr{tka}{q_1},
		\]
    so that clearly $U_k(q,a) = 0$ if $q_1 \nmid k$. Thus, suppose that $q_1 \mid k$, say $k = q_1k_1$.
    Noting that $k/q = k_1/p$, it follows that
		\<
			\label{eq:UkSumDiv2}
			U_k(q,a) = \fr1p \sum_{s=1}^p \xc^k(s) e\pfr{sk_1a}{p}.
		\>
	First, if $2\mid k$ then $\xc^k = \xc_0 \mmod{p}$ and the righthand side of \eqref{eq:UkSumDiv2} is
		\[
			\fr1p \sum_{s=1}^p \xc_0(s) e\pfr{sk_1a}{p} \begin{cases}
            1-1/p & p\mid k_1, \\
            -1/p & p\nmid k_1,
            \end{cases}
		\]
    which gives the first two cases of \eqref{eq:UkDiv}. If $2\nmid k$ then $\xc^k=\xc$, and the righthand side of \eqref{eq:UkSumDiv2} is equal to
		\[
			\fr1p \sum_{s=1}^p \xc(s) e\lf(\fr{sk_1a}{p}\rh) = \fr{G(\xc)\bar{\xc}(ak_1)}{p} = \fr{G(\xc)\xc(ak_1)}{p},
		\]
	establishing the final case of \eqref{eq:UkDiv}.
\end{proof}

\begin{lemma}
	Suppose that $p\mid q$ and $(a,q)=1$. If $2\nmid q$, then
        \[
            V_q(a) =\lf[\xc(a)G(\xc)(1-\tfr{\xc(2)}{4})L_2(\xc)/\fc - 1 \rh] \fr{p}{q^2},
        \]
    where $\fc = \tf{\pi^2}{24}(1-\tf{1}{p})$ and $L_2(\chi)=L(2,\chi)$.
\end{lemma}

\begin{proof}
	Let $q = pq_1$ for some $q_1\in\nn$. By Lemma \ref{lem:UkDiv}, $U_k(q,a)$ is nonzero only if $q_1\mid k$, so for such $k$ let $k = q_1k_1$. When $2\mid k$, since $2\nmid q$ one has $2\mid k$ if and only if $2\mid k_1$, so that
		\<
        \label{eq:2pEven}
			\sum_{2\ssmid k}^\infty \fr{U_k(q,a)}{k^2}
            = \sum_{2p \ssmid k_1}^\infty \fr{1}{(q_1k_1)^2} - \sum_{2\ssmid k_1}^\infty \fr{1/p}{(q_1k_1)^2} 
            = \fr{\pi^2}{24q^2}(1-p) = -\fr{\fc p}{q^2}.
		\>
	Now considering $\sum_{2\ssnmid k} U_k(q,a)/k^2$, by \eqref{eq:UkDiv} and the fact that $2\nmid k$ if and only if $2\nmid k_1$, we have
		\[
			\fr{U_k(q,a)}{k^2} = \fr{G(\xc)\xc(ak_1)}{p(q_1k_1)^2} = \fr{\xc(a)G(\xc)p}{q^2}\cdot\fr{\xc(k_1)}{k_1^2} \qquad (\text{for $2 \nmid k$}),
		\]
	so that
		\<
            \label{eq:2pOdd}
			\sum_{2 \ssnmid k}^\infty \fr{U_k(q,a)}{k^2}
			= \xc(a)G(\xc)\fr{p}{q^2} \sum_{2\ssnmid k_1}^\infty \fr{\xc(k_1)}{k_1^2} = \xc(a)G(\xc)(1-\tf{\xc(2)}{4})L_2(\xc)\cdot\fr{p}{q^2}.
		\>
    The result then follows from adding \eqref{eq:2pEven} and \eqref{eq:2pOdd} and dividing by $\fc$.
\end{proof}

Finally, using similar arguments to those above, we readily verify that
    \[
        V_q(a) = -\nfrac{4p}{q^2} \qquad \text{if $2p \mid q$},
    \]
which completes the verification of the formulae for $V_q(a)$ in Proposition \ref{prop:Major}.

Because many of our formulae involve the quantities $L_1(\xc_p)$, we conclude by including some formulae from \cite{montgomery2007multiplicative} that are useful for explicit computations. As usual, we write $L_r(\chi)=L(r,\chi)$ for $r = 0,1,2$.

\begin{lemma}[\cite{montgomery2007multiplicative}*{Thm.\ 9.9, Exer.\ 10.1.14, Exer.\ 10.1.15}]
    \label{lem:MV10.1.14}
    Suppose that $\xy$ is a primitive character modulo $q>1$ that $\xy(-1)=1$. One has $L_0(\xy)=0$,
        \begin{align*}
            L_1(\xy) &= - \fr{G(\xy)}{q} \sum_{a=1}^{q-1} \bar{\xy}(a) \log\sin\!\Big(\fr{\pi a}{q}\Big), \\
            L_2(\xy) &= \fr{\pi^2G(\xy)}{q} \sum_{a=1}^q \bar{\xy}(a)\lf(\pfr{a}{q}^2 - \fr{a}{q} + \fr{1}{6} \rh).
        \end{align*}
    If one has $\xy(-1)=-1$, then
        \begin{align*}
            L_0(\xy) = \fr{-1}{q} \sum_{a=1}^{q} a\xy(a) \qquad\text{and}\qquad L_1(\xy) = \fr{i\pi G(\xy)}{q^2} \sum_{a=1}^q a \bar{\xy}(a).
        \end{align*}
\end{lemma}

\bibliography{legendreBib.bib}
\end{document}